\documentclass[11pt,a4paper]{amsart}
\usepackage[latin1]{inputenc}
\usepackage{amsfonts}
\usepackage{amsmath,latexsym,amssymb,amsfonts, amsthm}
\usepackage{esint}
\usepackage{cite}
\usepackage{graphicx}
\usepackage{amscd}
\usepackage{color}
\usepackage{bm}             
\usepackage{enumerate}
\usepackage[dvips]{epsfig}
\usepackage{psfrag}
\usepackage{epsfig}

\usepackage[
  hmarginratio={1:1},     
  vmarginratio={1:1},     
  textwidth=15cm,        
  textheight=21cm,
  heightrounded,          
]{geometry}

\usepackage{graphicx,color}
\usepackage[colorlinks]{hyperref}
\hypersetup{linkcolor=blue,citecolor=blue,filecolor=black,urlcolor=blue}

\newtheorem{theorem}{Theorem}
\newtheorem{corollary}[theorem]{Corollary}
\newtheorem{proposition}[theorem]{Proposition}
\newtheorem{lemma}[theorem]{Lemma}
\theoremstyle{definition}
\newtheorem{remark}[theorem]{Remark}

\newtheorem{definition}[theorem]{Definition}

\numberwithin{theorem}{section}

\numberwithin{equation}{section}

\newcommand{\R}{\mathbb{R}}
\newcommand{\N}{\mathbb{N}}

\newcommand{\pa}{\partial}
\renewcommand{\S}{\mathbb{S}}

\newcommand{\cC}{{\mathcal C}}

\newcommand{\cF}{{\mathcal F}}   
   
\newcommand{\cH}{{\mathcal H}}

\newcommand{\cO}{{\mathcal O}}

\newcommand{\cR}{{\mathcal R}}   
\newcommand{\cS}{{\mathcal S}}   
\newcommand{\cT}{{\mathcal T}}   
   
\newcommand{\cV}{{\mathcal V}}

\newcommand{\dist}{{\rm dist}}

\newcommand{\weak}{\rightharpoonup}

\newcommand{\embed}{\hookrightarrow}

\newcommand{\eps}{\varepsilon}

\DeclareMathOperator{\loc}{loc}

\renewcommand{\epsilon}{\varepsilon}

\author[N. Soave]{Nicola Soave}\thanks{}
\address{Nicola Soave \newline \indent
Dipartimento di Matematica,  Politecnico di Milano,  \newline \indent
Via Edoardo Bonardi 9, 20133 Milano, Italy}
\email{nicola.soave@gmail.com; nicola.soave@polimi.it}

\author[S. Terracini]{Susanna Terracini}\thanks{}
\address{Susanna Terracini \newline \indent
 Dipartimento di Matematica ``Giuseppe Peano'', Universit\`a di Torino, \newline \indent
Via Carlo Alberto, 10,
10123 Torino, Italy}
\email{susanna.terracini@unito.it}

\title[The nodal set of solutions to some elliptic problems]{The nodal set of solutions to some elliptic problems: sublinear equations, and unstable two-phase membrane problem.}
\thanks{The authors are partially supported by the ERC Advanced Grant 2013 n. 339958 ``Complex Patterns for Strongly Interacting Dynamical Systems - COMPAT''. N. Soave is partially supported by the PRIN-2015KB9WPT\texttt{\char`_}010 Grant: ``Variational methods, with applications to problems in mathematical physics and geometry", and by the INDAM-GNAMPA project ``Aspetti non-locali in fenomeni di segregazione".}
\date{\today}

\begin{document}

\maketitle

\begin{abstract}
We are concerned with the nodal set of solutions to equations of the form 
\begin{equation*}
-\Delta u = \lambda_+ \left(u^+\right)^{q-1} - \lambda_- \left(u^-\right)^{q-1} \quad \text{in $B_1$}
\end{equation*}
where $\lambda_+,\lambda_- > 0$, $q \in [1,2)$, $B_1=B_1(0)$ is the unit ball in $\R^N$, $N \ge 2$, and $u^+:= \max\{u,0\}$, $u^-:= \max\{-u,0\}$ are the positive and the negative part of $u$, respectively. This class includes, the \emph{unstable two-phase membrane problem} ($q=1$), as well as \emph{sublinear} equations for $1<q<2$.  

We prove the following main results: (a) the finiteness of the vanishing order at every point and the complete characterization of the order spectrum; (b) a weak non-degeneracy property; (c) regularity of the nodal set of any solution: the nodal set is a locally finite collection of regular codimension one manifolds up to a residual singular set having Hausdorff dimension at most $N-2$ (locally finite when $N=2$); (d) a partial stratification theorem.

Ultimately, the main features of the nodal set are strictly related with those of the solutions to linear (or superlinear) equations, with two remarkable differences. First of all, the admissible vanishing orders can not exceed the critical value $2/(2-q)$. At threshold, we find a multiplicity of homogeneous solutions, yielding the \emph{non-validity} of any estimate of the $(N-1)$-dimensional measure of the nodal set of a solution in terms of the vanishing order.

The proofs are based on monotonicity formul\ae \ for a $2$-parameter family of Weiss-type functionals, blow-up arguments, and the classification of homogenous solutions. 
\end{abstract}

\section{Introduction}

In this paper we investigate the structure of the nodal set of solutions to 
\begin{equation}\label{eq}
-\Delta u = \lambda_+ \left(u^+\right)^{q-1} - \lambda_- \left(u^-\right)^{q-1} \quad \text{in $B_1$}
\end{equation}
where $\lambda_+,\lambda_- > 0$, $q \in [1,2)$, $B_1=B_1(0)$ is the unit ball in $\R^N$, $N \ge 2$, and $u^+:= \max\{u,0\}$, $u^-:= \max\{-u,0\}$ are the positive and the negative part of $u$, respectively. The main feature of the problem stays in the fact that the right hand side is not locally Lipschitz continuous as function of $u$, and precisely has sublinear character for $1<q<2$, and discontinuos character for $q=1$. Our study is driven by two main motivations: the comparison with the structure of the nodal set of solutions to linear problems and the investigation of the free boundaries of unstable obstacle problems with two phases.

The study of the nodal set of solutions to second order linear (or superlinear) elliptic equations stimulated a very intense research, starting from the seminal contribution by T. Carleman regarding the validity of the strong unique continuation principle \cite{Car}. Many generalizations of Carleman's result are now available, we refer the reader to \cite{KocTat} and the references therein for a more detailed discussion. In a slightly different direction, researcher also analyzed the structure of the nodal sets from the geometric point of view. For a weak solution $v$ of class at least $C^1$ (this is often the case by regularity theory), the nodal set splits into a regular part where $\nabla v \neq 0$ and a singular (or critical) set where $v$ vanishes together with its gradient. The regular part is in fact locally a $C^1$ graph by the implicit function theorem, so that the study of the nodal set reduces to the study of its singular subset. The first results concerning the structure of the singular set are due to L. Caffarelli and A. Friedman \cite{CafFri79, CafFri85}, who proved the partial regularity of the nodal set of solutions to semilinear elliptic equations driven by the Laplacian and with linear or superlinear right hand side; that is, their singular set has Hausdorff dimension at most $N-2$. For more general equations with sufficiently regular coefficients, the partial regularity has been established by R. Hardt and L. Simon \cite{HarSim}, and by F. Lin \cite{Lin} with different methods. Besides the partial regularity, in \cite{CafFri85,HarSim} it is also proved that, for classical solutions with relatively high order derivatives, the nodal set is a countable union of subsets of sufficiently smooth $(N-2)$-dimensional manifolds. A similar structure also holds under weaker regularity assumptions, that is, for weak solutions of linear equations in divergence form with Lipschitz coefficients and bounded first and zero order terms, see \cite{Han94} by Q. Han. The above contributions provide a fairly complete scenario from a qualitative point of view.

From a quantitative point of view, we recall the results in \cite{DonFef, HarSim, Lin} where the authors estimated the $(N-1)$-dimensional measure of the zero level set. Assuming that $v$ is a solution of a linear elliptic equations with analytic coefficients in $B_1$, in \cite{Lin} it is showed that the $(N-1)$-dimensional measure of $\{v=0\} \cap B_{1/2}$ can be estimated in terms of the vanishing order of $v$ in $0$, with a linear (optimal) proportional factor; see also \cite{DonFef}, where the analogue estimate was proved for eigenfunctions on analytic manifolds. For solutions to some linear equations with Lipschitz coefficients and bounded first and zero order terms, the best known result is contained in \cite{HarSim}. The above bounds can equivalently be formulated in terms of the \emph{frequency}
\begin{equation}\label{cap Lambda}
\Lambda:= \frac{\int_{B_1} |\nabla u|^2}{\int_{S_1} u^2}.
\end{equation}
We also refer to \cite{CheNabVal, HanHarLin, NabVal} for results regarding the estimate of the $(N-2)$-dimensional measure of the singular set. 

All the results for linear equations easily extend to a wide class of superlinear equations of type $-\Delta u= f(x,u)$, provided that $f(x,s)$ is locally Lipschitz continuous in $s$, uniformly in $x$, that $f(x,0)=0$, and that $u \in L^\infty_{\loc}$. In this case, one can simply set $c(x) = f(x,u(x))/u(x) \in L^\infty_{\loc}$, and regard the superlinear equation as $-\Delta u=c(x) u$. The picture changes drastically when we switch to sublinear or discontinuous cases, when the above function $c$ is no more bounded, and could not live in any $L^p_{\loc}$ space. 

A word of caution must be entered at this point: it is clear that when dealing with sublinear equations of type
\begin{equation}\label{mod}
\Delta u = \lambda_+ \left(u^+\right)^{q-1} - \lambda_- \left(u^-\right)^{q-1} \quad \text{in $B_1$}, \quad \text{with $1 \le q <2$},
\end{equation}
(where the sign of the laplacian is opposite to ours), the features of the nodal set of solutions are substantially different in comparison with the linear case: dead cores appear and no unique continuation can be expected. Indeed, already the ODE $u'' = |u|^{q-2} u$ admits non-trivial solutions whose nodal set has arbitrarily large interior. In this context one may try to describe the structure and the regularity of the free boundary $\pa \{u =0\}$. When $1<q<2$, we refer to \cite{AltPhil,Bon} (one phase problem), \cite{FoSh} (two phases problem), and references therein to results in this direction, while for $q=1$ we observe that \eqref{mod} boils down to the \emph{two phase membrane problem} (also called \emph{two phase obstacle problem})
\[
\Delta u = \lambda_+\chi_{\{u>0\}} - \lambda_-\chi_{\{u<0\}},
\]
studied in \cite{ShWe, ShUrWe, Wei01, Ura01}, see also the excellent monograph \cite{PeShUr}. 

In contrast, very little is known about the structure of the nodal sets for our equation \eqref{eq} with $\lambda_+, \lambda_->0$. Recently, T. Weth and the first named author proved in \cite{SoWe} the validity of the unique continuation principle for every $1 \le q <2$: non-trivial solutions cannot vanish on an open subset of $B_1$. The proofs in \cite{SoWe} are based on the control of the oscillation of the Almgren frequency function for solutions with a dead core: eventually, any such solution must vanish identically. An alternative approach based on Carleman's estimate has been recently presented in \cite{Ru} by A. R\"uland. When $1<q<2$, in \cite{Ru} it is also shown that the strong unique continuation principle holds: non-trivial solutions cannot vanish at infinite order. Note that key tools in proving unique continuation in the linear case, namely Almgren's monotonicity formula (as used in the pioneering papers \cite{GaLi86, GaLi87}), or Carleman estimates (see \cite{Car}), are not applicable in a standard way in the sublinear and discontinuous ones, and have to be considerably adjusted (see \cite{Ru,SoWe}). It is in any case very natural to ask what are the other possible common properties of the nodal sets of solutions to $\eqref{eq}$, in comparison with solutions to linear equations.

As a further motivation, we observe the similarity of problem \eqref{eq} with obstacle-type problems in the case when $q=1$. As already mentioned,  \eqref{eq} becomes a two-phase obstacle-type problem with the ``wrong" sign. The presence of the minus in front of the Laplacian modifies completely the structure of the problem; let us consider for instance the so called \emph{unstable obstacle problem} 
\begin{equation}\label{unst}
-\Delta u =  \chi_{\{u>0\}},
\end{equation}
studied in \cite{AnShWe1,AnShWe2, AnShWe3,AnWe,MoWe} (this corresponds to $q=1$ and $\lambda_-=0$ in \eqref{eq}). In these contributions, the main differences between the study of the classical (stable) obstacle-type problems and the unstable ones are putted in evidence: J. Andersson and G. S. Weiss proved that solutions of \eqref{unst} do not achieve the optimal $C^{1,1}$ regularity in general, and can be degenerate at free boundary points, see \cite{AnWe}. This fact prevents the use of several classical methods. Despite these obstructions, R. Monneau and G. S. Weiss proved the partial regularity of non-degenerate solutions to \eqref{unst}, and the smoothness of the nodal set of energy minimizers in dimension $N=2$ in \cite{MoWe}; afterwards, J. Andersson, H. Shahgholian and G. S. Weiss established existence and uniqueness of non-trivial homogeneous blow-ups at non-degenerate singular points in dimension $N=2$ \cite{AnShWe1} and $N = 3$ \cite{AnShWe2}, deriving as a consequence the geometric structure of the non-degenerate singular set; the structure of the codimension $2$ non-degenerate singular set for arbitrary $N \ge 4$ has been investigated in \cite{AnShWe3}. 

As far as we know, for a generic solution  to \eqref{unst} (not necessarily non-degenerate, nor minimal) the partial regularity of the nodal set, or the finiteness of the admissible vanishing orders, are still open problems (see the Open Questions in \cite{AnWe}).

In this paper we deal with the two phases problem \eqref{eq}, treating simultaneously the case $q=1$, which we call \emph{unstable two phase membrane problem} in analogy with \eqref{unst}, and the case $1<q<2$, a prototype of sublinear equation, proving the following main results:
\begin{itemize}
\item the finiteness of the vanishing order at every point;
\item the characterization of all the admissible vanishing orders for solution to \eqref{eq};
\item the non-degeneracy property of any solution;
\item the regularity of the nodal set of any solution: the nodal set is a locally finite collection of regular codimension one manifolds up to a residual singular set having Hausdorff dimension at most $N-2$;
\item a partial stratification for the nodal set;
\item a multiplicity result, yielding the \emph{non-validity} of any estimate of the $(N-1)$-dimensional measure of the nodal set of a solution in terms of the vanishing order in a zero.
\end{itemize}

As a byproduct of our method, we prove that the finiteness of vanishing order also holds for $\lambda_-=0$, thus answering an open question raised in \cite{AnWe}. In a forthcoming paper \cite{SoTeforth} we shall treat also the singular case $0<q<1$. 

In the next subsection we state our results in a precise form, and introduce the notation and the terminology which will be used throughout the paper.

\subsection{Statement of the main results}

Let $B_1= B_1(0)$ denote the ball of center $0$ and radius $1$ in $\R^N$, $N \ge 2$, and let $\lambda_+, \lambda_- \ge 0$ (most of the paper will actually deal with the case $\lambda_+,\lambda_- >0$). We consider weak solutions $u \in H^1_{\loc}(B_1)$ of the the second order equation \eqref{eq}, and we describe the structure of the zero level set $Z(u) :=u^{-1}(\{0\}) \subset B_1$. By standard elliptic regularity, any weak solution is of class $C^{1,\alpha}(B_1)$ for every $\alpha \in (0,1)$. If $q >1$, then weak solutions are in fact classical $C^2$ solutions, but since we shall address simultaneously also the case $q=1$ we will never use this fact. From now on, we simply write ``solution" insted of ``weak solution", for the sake of brevity.

We define the regular part $\cR(u) \subset Z(u)$ and the singular part $\Sigma(u) \subset Z(u)$ by
\begin{align*}
\cR(u):=\{ x \in B_1: u(x) = 0 \text{ and } \nabla u(x) \neq 0\}, \quad \Sigma(u):= \{ x \in B_1: u(x) = |\nabla u(x)| = 0\};
\end{align*}
$\cR(u)$ is in fact locally a $C^{1,\alpha}$ $(N-1)$-dimensional hypersurface by the implicit function theorem. 

\begin{definition}\label{def: order V}
Let $u$ be a solution to \eqref{eq}, and let $x_0 \in Z(u)$. The \emph{vanishing order of $u$ in $x_0$} is defined as the number $\cV(u,x_0) \in \R^+$ with the property that 
\[
\limsup_{r \to 0^+} \frac1{r^{N-1+2\beta}} \int_{S_r(x_0)} u^2  = \begin{cases} 0 & \text{if $0 <\beta< \cV(u,x_0)$} \\
+\infty  & \text{if $\beta > \cV(u,x_0)$}.
\end{cases}
\]  
If no such number exists, then 
\[
\limsup_{r \to 0^+} \frac1{r^{N-1+2 \beta}} \int_{S_r(x_0)} u^2 = 0 \quad \text{for any $\beta>0$},
\]
and we set $\cV(u,x_0)=+\infty$. 
\end{definition}

Here and in what follows, for $x_0 \in B_1$ and $0<r<\dist(x_0,\pa B_1)$, we let $S_r(x_0) = \pa B_r(x_0)$, where $B_r(x_0)$ is the ball of center $x_0$ and radius $r$. In the frequent case $x_0=0$, we simply write $B_r$ and $S_r$ for the sake of brevity. 

The $\limsup$ appearing in the Definition \ref{def: order V} describes the growth of $u$ on spheres $S_r(x_0)$ of varying radii. Other definitions of vanishing order could have been possible, see e.g. \cite{Han94, HarSim}. For linear and superlinear elliptic equations in divergence form, it can be showed that all of them coincide. Moreover, in such cases the strong unique continuation and the existence of a harmonic profile near each point of the zero level set \cite{CafFri85, Han94} ensure that the vanishing order is finite, and can be any positive integer. As we shall see, this is not the case for the sublinear equation \eqref{eq}. So far it was only known that, if $\lambda_+,\lambda_->0$ and $1 \le q <2$, then the set $Z(u)$ has empty interior whenever $u \not \equiv 0$ \cite{SoWe}, and that if in addition $1<q<2$, then the vanishing order is finite, see \cite{Ru}\footnote{In \cite{Ru}, a different notion of vanishing order is used, in terms of the quantity $\int_{B_r(x_0)} u^2$. As a result, the version of the strong unique continuation principle in \cite{Ru} for $\lambda_+,\lambda_->0$ and $1<q<2$ is slightly stronger than ours.}. Our first main result establishes the validity of the strong unique continuation principle for every $1 \le q <2$, for every $\lambda_+>0$ and $\lambda_- \ge 0$.

\begin{theorem}[Strong unique continuation]\label{thm: strong intro}
Let $1 \le q <2$, $\lambda_+>0$, $\lambda_-\ge0$, $u \in H^1_{\loc}(B_1)$ solve \eqref{eq}, and let $x_0 \in Z(u)$. If $\cV(u,x_0)=+\infty$, then necessarily $u \equiv 0$; in particular, if for every $\beta >0$ it results that
\[
\lim_{|x-x_0| \to 0^+} \frac{|u(x)|}{|x-x_0|^\beta} =0,
\]
then necessarily $u \equiv 0$.
\end{theorem}

For $q=1$ and $\lambda_-=0$, this answers an open question raised in \cite{AnWe}.

When both $\lambda_+$ and $\lambda_-$ are positive, we can prove a much better result, characterizing all the admissible vanishing orders. Let $\beta_q \in \N$ be the larger positive integer strictly smaller than $2/(2-q)$:
\begin{equation}\label{def beta q}
\beta_q:= \begin{cases}
\left\lfloor \frac{2}{2-q} \right\rfloor & \text{if } \frac{2}{2-q} \not \in \N \\
\frac{2}{2-q}-1 & \text{if } \frac{2}{2-q} \in \N.
\end{cases}
\end{equation}

\begin{theorem}[Classification of the vanishing orders]\label{thm: very strong V}
Let $1 \le q <2$, $\lambda_+,\lambda_->0$, $0 \not \equiv u \in H^1_{\loc}(B_1)$ solve \eqref{eq}, and let $x_0 \in Z(u)$. Then 
\[
\cV(u,x_0) \in \left\{1,\dots, \beta_q, \frac2{2-q} \right\}.
\]
In particular, if $q=1$ then $\cV(u,x_0) \in \{1,2\}$.
\end{theorem}

We also have an important non-degeneracy property.

\begin{theorem}[Non-degeneracy]\label{thm: non-deg V}
Let $1 \le q <2$, $\lambda_+,\lambda_->0$, $0 \not \equiv u \in H^1_{\loc}(B_1)$ solve \eqref{eq}, and let $x_0 \in Z(u)$. Then 
\[
\liminf_{r \to 0^+} \frac{1}{r^{N-1+2 \cV(u,x_0)}} \int_{S_r(x_0)} u^2>0.
\]
\end{theorem}

\begin{remark}\label{rmk: differences}
Theorem \ref{thm: very strong V} reveals a deep difference between linear and sublinear equations: while for the formers solutions can vanish at any integer order, for the latters we have a universal bound, depending only on $q$, on the admissible vanishing orders. \\
Theorem \ref{thm: non-deg V} reveals moreover a striking difference between the case $\lambda_-=0$ and the one $\lambda_->0$. Indeed, if $\lambda_-=0$ and $q=1$ degenerate solutions do exist, as proved in \cite{AnWe}. The reason behind this discrepancy ultimately rests in the presence of a large set of global homogeneous solutions to \eqref{eq} for $\lambda_->0$, which does not exist when $\lambda_-=0$. When $\lambda_-=0$, for $q=1$ there exists only one $2$-homogeneous solution to \eqref{eq} in $\R^2$, up to rotations, and this was the key in the proof of \cite{AnWe} (we refer to \cite{AnShWe2}, Remark 3.3 for more details). In contrast, we shall prove that for every $\lambda_+,\lambda_->0$ and $q \in [1,2)$ problem \eqref{eq} admits infinitely many global $2$-homogeneous solutions, see Theorem \ref{thm: multiple} below.
\end{remark}

Next, we study the existence of blow-up limits around points of $Z(u)$. For linear (and superlinear) equations, it is known that solutions behave like harmonic polynomials in a neighborhood of each point of $Z(u)$, see \cite{Ber}, \cite[Theorem 1.2]{CafFri85} or \cite[Theorem 3.1]{Han94}. In the sublinear setting this is not necessarily the case.

\begin{theorem}[Blow-ups]\label{thm: blow-up}
Let $1 \le q<2$, $\lambda_+,\lambda_->0$, $0 \not \equiv u \in H^1_{\loc}(B_1)$ solve \eqref{eq}, $x_0 \in Z(u)$, $R \in (0, \dist(x_0,\pa B_1))$, and let $\beta_q$ be defined by \eqref{def beta q}. 

Then the following alternative holds:
\begin{itemize}
\item[($i$)] if $d_{x_0}:=\cV(u,x_0) \in \{1,\dots,\beta_q\}$, then there exist a homogeneous harmonic polynomial $P_{x_0} \not \equiv 0$ of degree $d_{x_0}$, and a function $\Gamma_{x_0}$, such that
\[
u(x) = P_{x_0}(x-x_0) + \Gamma_{x_0}(x) \qquad \text{in $B_R(x_0)$},
\]
with 
\[
\begin{cases} \Gamma_{x_0}(x)| \le C |x-x_0|^{d_{x_0}+\delta}  \\ 
|\nabla \Gamma_{x_0}(x)| \le C |x-x_0|^{d_{x_0}-1+\delta} \end{cases}
 \qquad \text{in $B_R(x_0)$}
\]
for suitable constants $C,\delta>0$;
\item[($ii$)] if $\cV(u,x_0) = 2/(2-q)$, then for every sequence $0<r_n \to 0^+$ we have, up to a subsequence,
\[
\frac{u(x_0+r_n x)}{\left( \frac{1}{r^{N-1}} \int_{S_r(x_0)} u^2\right)^\frac12}
 \to \bar u \qquad \text{in $C^{1,\alpha}_{\loc}(\R^N)$ for every $0<\alpha<1$},
\]
where $\bar u$ is a $2/(2-q)$-homogeneous non-trivial solution to 
\[
-\Delta \bar u= \mu \left( \lambda_+ (\bar u)^{q-1} -\lambda_- (\bar u)^{q-1} \right) \quad \text{in $\R^N$}
\]
for some $\mu \ge 0$. Moreover, the case $\mu = 0$ is possible only if $2/(2-q) \in \N$. 
\end{itemize}
\end{theorem} 

If alternative ($ii$) takes place in Theorem \ref{thm: blow-up}, one may wonder if the existence of the blow-up limit could be replaced by a full Taylor expansion, as in point ($i$). The results in \cite{AnShWe1} and \cite{AnShWe2} suggest that this could be the case, but the expansion should be by far more involved.

In any case, Theorem \ref{thm: blow-up} allows to estimate the Hausdorff dimension of the singular set, via the dimension reduction principle due to Federer.

\begin{theorem}[Hausdorff dimension of nodal and singular set]\label{thm: Hausdorff}
Let $1 \le q<2$, $\lambda_+,\lambda_->0$, and let $0 \not \equiv u \in H^1_{\loc}(\Omega)$ be a solution of \eqref{eq}. The nodal set $Z(u)$ has Hausdorff dimension $N-1$, and the singular set $\Sigma(u)$ has Hausdorff dimension at most $N-2$. Furthermore, if $N=2$ the singular set $\Sigma(u)$ is discrete.
\end{theorem}

So far, we studied the asymptotic properties of solutions to \eqref{eq} near their zero level set $Z(u)$, and derived an estimate on the dimension of $Z(u)$ and of its singular subset $\Sigma(u)$. Now we study more in details the geometric structure of $\Sigma(u)$ in case $1<q<2$. Inspired by the results available in the linear case \cite{Han94}, it would be natural to conjecture that $\Sigma(u)$ is countably $(N-2)$-rectifiable. Two main ingredients are fundamental in order to prove such a result: firstly, the uniqueness of homogeneous blow-ups at singular points, and, secondly, the upper semi-continuity of the vanishing order map $x_0 \in Z(u) \mapsto \cV(u,x_0)$. The uniqueness of the homogeneous blow-ups represents an obstacle we could not overcome. With regard to this, we mention that a Monneau's monotonicity formula seems not available in our setting, and moreover \eqref{eq} posed in the all space $\R^N$ admits infinitely many geometrically distinct (that is, they cannot be obtained one by the other with rotations or scalings) $2/(2-q)$-homogeneous solutions, see Theorem \ref{thm: multiple} below. It is possible that a sophisticated Fourier expansion, such as those in \cite{AnShWe1, AnShWe2}, finally lead to uniqueness, but at the moment we leave this problem as open. As far as the upper semi-continuity of the vanishing order map is concerned, for linear elliptic equations it follows easily by Almgren's monotonicity formula, which is not available in our context. 

With an alternative approach, we can however partially restore the $(N-2)$-rectifiability of $\Sigma(u)$. We split the singular set into its ``good" part
\[
\mathcal{S}(u):=\{x_0 \in Z(u): \ 2 \le \cV(u,x_0) \le \beta_q
\},
\]
and its ``bad" part 
\[
\mathcal{T}(u):=\left\{x_0 \in Z(u): \ \cV(u,x_0) = \frac{2}{2-q}
\right\}.
\]
For $x_0 \in \mathcal{S}(u)$, the function $P_{x_0}$ is called the \emph{leading polynomial of $u$ at $x_0$}. Then we have:

\begin{theorem}[Partial stratification]\label{thm: struct sing}
Let $1 < q<2$, $\lambda_+,\lambda_->0$, and let $0 \not \equiv u \in H^1_{\loc}(\Omega)$ be a solution of \eqref{eq}. Then $\mathcal{S}(u)$ is an $(N-2)$-dimensional countably rectifiable set. More precisely, the decomposition
\[
\mathcal{S}(u) = \bigcup_{j=0}^{N-2} \mathcal{S}^j(u)
\]
holds true, where each $\mathcal{S}^j(u)$ is on a countable union of $j$-dimensional Lipschitz-graphs for $j=0,\dots,N-3$, and $\mathcal{S}^{N-2}(u)$ is on a countable union of $(N-2)$-dimensional $C^{1,\alpha}$-graphs for some $0<\alpha<1$. Furthermore, the set $\mathcal{T}(u)$ is relatively closed in $\Sigma(u)$, which is relatively closed in $Z(u)$.
\end{theorem}

\begin{remark}
Theorems \ref{thm: Hausdorff} and \ref{thm: struct sing} describe the free boundary regularity without knowing if the optimal regularity of solutions to \eqref{eq} (i.e., $u \in C^{1,1}$ if $q=1$, and $u \in C^{2,(q-1)}$ if $1<q<2$) is achieved; the results regarding \eqref{unst} suggest that this could not be the case for an arbitrary solution. 
\end{remark}

As last issue, we analyze the size of the nodal set. Having already recalled the fundamental results in \cite{DonFef, HarSim, Lin} regarding linear equations, we address the following question: for solutions to \eqref{eq} vanishing in $0$, is it true that
\[
\cH^{N-1}(Z(u) \cap B_{1/2}) \le f( \cV(u,0)), \quad \text{or} \quad \cH^{N-1}(Z(u) \cap B_{1/2}) \le g\left(\frac{\int_{B_1} |\nabla u|^2-|u|^q}{ \int_{S_1} u^2} \right),
\]
for some $f, g$ positive and monotone increasing? Notice that the integral appearing as the argument of $g$ is the natural analogue of the quantity $\Lambda$ defined in \eqref{cap Lambda}. 

We show that \emph{no bound of the previous form can exist}. 

\begin{theorem}\label{thm: multiple}
Let $1  \le q<2$ and $\lambda_+,\lambda_->0$. There exists $\bar k \in \N$ depending only on $q$ such that, if $k>\bar k$ is an integer, then equation \eqref{eq} has a global $2/(2-q)$-homogeneous solution $u_k$ with the following properties: 
\begin{itemize}
\item[($i$)] it results
\begin{equation}\label{O e N}
\cO(u_k,0) = \frac{2}{2-q} = \frac{\int_{B_1} |\nabla u_k|^2-|u_k|^q}{ \int_{S_1} u_k^2};
\end{equation}
\item[($ii$)] the nodal set $Z(u_k)$ is the union of $2k$ straight lines passing through the origin.
\end{itemize} 
In particular, by point ($ii$)
\[
\cH^{N-1}(Z(u_k) \cap B_{1/2}) \to +\infty
\]
as $k \to \infty$. 
\end{theorem}

\begin{remark}
Beyond the non-validity of estimates involving the $(N-1)$-dimensional measure of $Z(u)$, the previous result also shows that for $\lambda_+,\lambda_->0$ the set of geometrically distinct global $2/(2-q)$-homogeneous solutions to \eqref{eq} is very rich, even in case $q=1$. As already mentioned in Remark \ref{rmk: differences}, this marks a remarkable difference between the cases $\lambda_+,\lambda_->0$ and $\lambda_+>0, \lambda_-=0$.
\end{remark}
  
\paragraph{\textbf{Structure of the paper.}} In Section \ref{sec: pre} we prove some preliminary results which will be frequently used throughout the rest of the paper: we initiate the study of the local behavior of solutions to \eqref{eq} near nodal points, showing that either point ($i$) in Theorem \ref{thm: blow-up} holds for $u$ in $x_0$, or $u$ must decay sufficiently fast; then, we introduce the fundamental object of our analysis, a two-parameters family of Weiss-type functionals 
\begin{align*}
W_{\gamma,t}(u,x_0,r) = \frac{1}{r^{N-2+2\gamma}} \int_{B_r(x_0)} &\left(|\nabla u|^2 - \frac{t}q \Big(\lambda_+ (u^+)^q + \lambda_-(u^-)^q \Big) \right)\,dx \\
& - \frac{\gamma}{r^{N-1+2 \gamma}}  \int_{S_r(x_0)} u^2 \, d\sigma,
\end{align*}
and derive the expression of the derivatives of such functionals. Notice the presence of the two parameters $t$ and $\gamma$, which will play a crucial role in our argument.

With the monotonicity formulae in our hands, we prove the strong unique continuation principle, Theorem \ref{thm: strong intro}, in Section \ref{sec: strong}. 

Afterwards, we always focus on the case $\lambda_+,\lambda_- >0$, and we address the classification of the admissible vanishing orders, and the non-degeneracy of the solutions. At a first stage, it is convenient to work with a different notion of vanishing order. We set
\begin{equation}\label{def norm}
\|u\|_{x_0,r}:= \left( \frac1{r^{N-2}} \int_{B_r(x_0)} |\nabla u|^2\, dx + \frac{1}{r^{N-1}} \int_{S_r(x_0)} u^2 \, d\sigma\right)^\frac12.
\end{equation}
For any $0<r<\dist(x_0,\pa B_1)$ fixed, this is a norm in $H^1(B_r(x_0))$, equivalent to the standard one by trace theory and Poincar\'e's inequality.

\begin{definition}\label{def: order}
For a solution $u$ to \eqref{eq}, let $x_0 \in Z(u)$. The $H^1$-\emph{vanishing order of $u$ in $x_0$} is defined as the number $\cO(u,x_0) \in \R^+$ with the property that 
\[
\limsup_{r \to 0^+} \frac1{r^{2\beta}} \|u\|_{x_0,r}^2  = \begin{cases} 0 & \text{if $0 <\beta< \cO(u,x_0)$} \\
+\infty  & \text{if $\beta > \cO(u,x_0)$}.
\end{cases}
\]  
If no such number exists, then 
\[
\limsup_{r \to 0^+} \frac1{r^{2 \beta}} \|u\|_{x_0,r}^2 = 0 \quad \text{for any $\beta>0$},
\]
and we set $\cO(u,x_0)=+\infty$. 
\end{definition}
We will prove then two variants of Theorems \ref{thm: very strong V} and \ref{thm: non-deg V}: that is, we prove the very same statements with the quantity $\cV$ replaced by $\cO$; this is the content of Section \ref{sec: strong and non-deg}. The advantage of working with $\cO$ stays in the fact that, involving the $H^1$-norm, we have better control of the behavior of solutions.

In Section \ref{sec: blow-up} we discuss the existence of homogeneous blow-up limits for sequences of type
\[
\frac{u(x_0+r x)}{\|u\|_{x_0,r}},
\]
as $r \to 0^+$, see Theorem \ref{thm: blow-up '}. The main ingredients in the blow-up analysis are the non-degeneracy of the solutions and the upper semi-continuity of the vanishing order map along converging sequences of solutions. The proof of this last fact is given in Proposition \ref{prop: upper} which, being quite long, is the object of Section \ref{sec: upper}. As byproduct of Theorem \ref{thm: blow-up '}, we shall also recover Theorems \ref{thm: very strong V}, \ref{thm: non-deg V} and \ref{thm: blow-up} as stated in the introduction (that is, in terms of the order $\cV$).

In Section \ref{sec: structure} we give the description of the nodal set, proving Theorems \ref{thm: Hausdorff} and \ref{thm: struct sing}.

Finally, in Section \ref{sec: global} we prove our multiplicity result, Theorem \ref{thm: multiple}.

\section{Preliminaries}\label{sec: pre}

\subsection{A partial blow-up.}

In this subsection we give a first insight at the behavior of $u$ close to $Z(u)$, proving the first half of Theorem \ref{thm: blow-up}.

\begin{proposition}\label{thm: blow-up pre}
Let $1 \le q<2$, $u \in H^{1}_{\loc}(B_1)$ be a solution of \eqref{eq}, $x_0 \in Z(u)$, $R  \in (0, \dist(x_0,\pa B_1))$, and let $\beta_q$ be defined by \eqref{def beta q}.

Then the following alternative holds:
\begin{itemize}
\item[($i$)] either there exist $d_{x_0} \in \N$, $1 \le d_{x_0} \le \beta_q$, a homogeneous harmonic polynomial $P_{x_0} \not \equiv 0$ of degree $d_{x_0}$, and a function $\Gamma_{x_0}$ such that
\[
u(x) = P_{x_0}(x-x_0) + \Gamma_{x_0}(x) \qquad \text{in $B_R(x_0)$},
\]
with 
\[
\begin{cases} \Gamma_{x_0}(x)| \le C |x-x_0|^{d_{x_0}+\delta}  \\ 
|\nabla \Gamma_{x_0}(x)| \le C |x-x_0|^{d_{x_0}-1+\delta} \end{cases}
 \qquad \text{in $B_R(x_0)$}
\]
for suitable constants $C,\delta>0$;
\item[($ii$)] or, for every $\eps>0$, there exists $C_\eps>0$ such that
\[
\begin{cases}
|u(x)| \le C_\eps |x-x_0|^{\frac2{2-q}-\eps} \\    |\nabla u(x)| \le C_\eps |x-x_0|^{\frac{2}{2-q}-1-\eps}\end{cases} \qquad \text{in $B_R(x_0)$}.
\]
\end{itemize}
\end{proposition}

When $q=1$, the proposition simply says that either $\nabla u(x_0) \neq 0$, or for every $0<\alpha<1$ there exists $C_\alpha>0$ such that
\[
\begin{cases}
|u(x)| \le C_\alpha |x-x_0|^{1+\alpha} \\    |\nabla u(x)| \le C_\alpha |x-x_0|^{\alpha}\end{cases} \qquad \text{in $B_\rho(x_0)$}.
\]
This follows as direct consequence of the $C^{1,\alpha}$ regularity of $u$ (for every $0<\alpha<1$).

Let us focus then on $1<q<2$; the proof is based upon an iterated application of \cite[Lemma 3.1]{CafFri79} - \cite[Lemma 1.1]{CafFri85}. For the reader's convenience, we anticipate the following simple lemma.

\begin{lemma}\label{lem: beta_k}
Let $1<q<2$, and for any $k \in \N$ and arbitrary $\delta_k \in \left[0,\frac{1}{2^k}\right)$, let us set
\begin{equation}\label{def beta k delta}
\begin{cases}
\beta_{1}:= (q+1) \\
\beta_{k}:= (q-1) \beta_{k-1} + 2 -\delta_k.
\end{cases}
\end{equation}
It is possible to choose the sequence $\{\delta_k\}$ in such a way that 
\begin{equation}\label{lim beta_k}
\beta_{k} \not \in \N \quad \text{for every $k$, and} \quad \beta_{k} \nearrow \frac{2}{2-q} \quad \text{as $k \to \infty$}.
\end{equation}
\end{lemma}
\begin{proof}
At first, we claim that, independently on the choice of $\delta_k \in \left[0,\frac{1}{2^k}\right)$, we have
\begin{equation}\label{beta minor}
\beta_{k} < \frac{2}{2-q} \qquad \text{for every $k \in \N$}.
\end{equation}
To prove the claim, we note that if $\beta_{k} < \frac2{2-q}$, then in turn 
\[
\beta_{k+1} \le (q-1) \beta_{k} +2 < \frac{2(q-1)}{2-q}+2 = \frac2{2-q}.
\]
Therefore, claim \eqref{beta minor} follows once that we have checked that $\beta_1 < 2/(2-q)$,
and this is clearly verified for any $q \in (1,2)$. 

Having established \eqref{beta minor}, we claim that it is possible to choose $\delta_k \in \left[0,\frac1{2^k}\right)$ such that 
\begin{equation}\label{beta mon}
\text{$\beta_k \not \in \N$ for every $k$, and $\{\beta_{k}\}$ is a monotone increasing sequence}.
\end{equation}
It is sufficient to let $\delta_1=0$, and then take any
\[
0 \le  \delta_k < \min\left\{\frac1{2^k}, 2-(2-q) \beta_{k-1}\right\}
\]
such that $\beta_k \not \in \N$. Such a choice of $\delta_k$ is possible, by \eqref{beta minor}, and it is immediate to verify that $\beta_{k} > \beta_{k-1}$ for every $k$.

By \eqref{beta minor} and \eqref{beta mon}, $\beta_{k}$ tends to some $\bar \beta \in \R$ as $k \to \infty$, and, passing to the limit into \eqref{def beta k delta}, we deduce that $\bar \beta = 2/(2-q)$.
\end{proof}

\begin{proof}[Proof of Proposition \ref{thm: blow-up pre}]
Throughout this proof $\delta_k$ and $\beta_k$ denote the numbers defined in Lemma \ref{lem: beta_k}. Without loss of generality, we can suppose that $x_0=0$ and we take $R \in (0,1)$. By elliptic regularity, $u \in C^{1,\alpha}(\overline{B_R})$, and hence there exists $L>0$ such that 
\[
|\Delta u(x)| \le \max\{\lambda_+,\lambda_-\} |u(x)|^{q-1} \le \max\left\{ \max\{\lambda_+,\lambda_-\} L^{q-1}, 2^{q-1}\right\} |x|^{q-1} \qquad \text{in $B_R$}.
\]
Thus, by \cite[Lemma 1.1]{CafFri85} there exists a harmonic polynomial $P_1$ of degree $\lfloor q-1 \rfloor+2 =2$ and a function $\Gamma_1$ such that 
\[
u(x) = P_1(x) + \Gamma_1(x) \qquad \text{in $B_R$},
\]
with 
\[
|\Gamma_1(x) | \le C_1  |x|^{q+1}, \quad |\nabla \Gamma_1(x)| \le C_1 |x|^{q} \quad \text{in $B_R$},
\]
for a positive constant $C_1>0$ depending only on $q$, $\|u\|_{W^{1,\infty}(S_R)}$ and $N$. Now, if $P_1 \not \equiv 0$, the proof is complete. If instead $P_1 \equiv 0$, then $u = \Gamma_1$, and hence by the above estimates
\begin{align*}
|\Delta u(x) | & \le \max\{\lambda_+,\lambda_-\} |u(x)|^{q-1} \le \max\{\lambda_+,\lambda_-\} C_1^{q-1} |x|^{(q-1)(q+1)} \\
& \le \max\left\{ \max\{\lambda_+,\lambda_-\} C_1^{q-1}, 2^{(q-1)(q+1)-\delta_2}\right\} |x|^{(q-1)(q+1) - \delta_2}  \qquad \text{in $B_R$},
\end{align*}
with $(q+1)(q-1) -\delta_2 = \beta_2-2 \not \in \N$. As a consequence, we can apply again \cite[Lemma 1.1]{CafFri85}: letting
\[
\alpha_2:= \lfloor (q-1)(q+1) - \delta_2 \rfloor+2,
\]
there exist a harmonic polynomial $P_2$ of degree $\alpha_2$, a function $\Gamma_2$, and a constant $C_2>0$ depending on the data such that
\[
u(x) = P_2(x) + \Gamma_2(x) \qquad \text{in $B_R$},
\]
and
\[
|\Gamma_2(x) | \le C_2  |x|^{\beta_2}, \quad |\nabla \Gamma_2(x)| \le C_2 |x|^{\beta_2-1} \quad \text{in $B_R$}.
\] 
If $P_2 \not \equiv 0$, then the proof is complete. If instead $P_2 \equiv 0$, then $u=\Gamma_2$ and we can iterate the previous argument in the following way: for any $k \ge 3$ such that $P_{k-1} \equiv 0$, we let
\begin{equation}\label{def beta_k}
\alpha_k:= \lfloor (q-1)\beta_{k-1} - \delta_k \rfloor+2;
\end{equation}
then there exist a harmonic polynomial $P_k$ of degree $\alpha_k$, a function $\Gamma_k$, and a constant $C_k>0$ depending on the data such that
\[
u(x) = P_k(x) + \Gamma_k(x) \qquad \text{in $B_R$},
\]
and
\[
|\Gamma_k(x) | \le C_k  |x|^{\beta_k}, \quad |\Gamma_k(x) | \le C_k  |x|^{\beta_k-1}\quad \text{in $B_R$}.
\] 
Since $\beta_k \nearrow \frac2{2-q}$, we deduce that either there exists a minimum integer $m \in \{1,\dots,\beta_q\}$ (with $\beta_q$ defined by \eqref{def beta q}) such that $P_m \not \equiv 0$, or else for any fixed $\eps$ there exists $k \in \N$ with $\frac2{2-q} -\eps \le \beta_k< \frac2{2-q}$, and for such index $k$ we have
\[
|u(x)| = |\Gamma_k(x)|  \le C_k |x|^{\frac2{2-q}-\eps}, \quad |\nabla u(x) | = |\nabla \Gamma_k(x)| \le C_k |x|^{\frac2{2-q}-\eps-1} . \qedhere
\]
\end{proof}

\subsection{Almgren and Weiss type functionals} 

Let $\Omega \subset \R^N$ be a domain, $\mu_+>0$, $\mu_- \ge 0$, and let $v \in H^1_{\loc}(\Omega)$ be a weak solution to 
\begin{equation}\label{eq mu}
-\Delta v = \mu_+ (v^+)^{q-1} -\mu_- (v^-)^{q-1} \quad \text{in $\Omega$}, 
\end{equation}
with $1 \le q<2$. 
Let 
\[
F_{\mu_+, \mu_-}(v):= \mu_+ (v^+)^{q} +\mu_- (v^-)^{q}.
\]
For $x_0 \in \Omega$, $0<r<\dist(x_0,\pa \Omega)$, and for $\gamma,t>0$, we consider the functionals:
\begin{align*}
H(v,x_0, r) &:= \int_{S_r(x_0)} v^2 \, d\sigma,  \\
D_t(v,x_0,r) &:= \int_{B_r(x_0)}\left(|\nabla v|^2 - \frac{t}q F_{\mu_+, \mu_-}(v)\right)\,dx, \\
N_t(v,x_0,r) &:= \frac{r D_t(v,x_0,r)}{H(v,x_0,r)}, \quad \text{defined provided that $H(v,x_0,r) \neq 0$},\\
W_{\gamma,t}(v,x_0,r)&:= \frac{1}{r^{N-2+2\gamma}} D_t(v,x_0,r) - \frac{\gamma}{r^{N-1+2\gamma}}H(v,x_0,r).
\end{align*}
The definitions of $F$, $D_t$, $N_t$ and $W_{\gamma,t}$ involves also the quantities $\mu_+$ and $\mu_-$. Since these will always be uniquely determined by $v$ via \eqref{eq mu}, we do not stress this dependence. The functions $N_t$ and $W_{\gamma,t}$ are an \emph{Almgren-type frequency} and a \emph{Weiss-type functional}, respectively. 

In what follows, we fix $v$ solution to \eqref{eq mu}, and we recall or derive several relations involving the above quantities and their derivatives, which will be frequently used throughout the paper. As usual, $\nu$ denotes the outer unit normal vector on $S_r(x_0)$, and $\pa_\nu u = u_\nu$ denotes the outer normal derivative. We shall often omit the volume and area elements $dx$ and $d\sigma$. 

By definition and using the divergence theorem\footnote{In case $q=1$ the classical divergence theorem is not applicable, but we can appeal to more general versions such as \cite[Proposition 2.7]{HoMiTa}.}, we have
\begin{equation}\label{D bou}
D_t(v,x_0,r) = \int_{S_r(x_0)} v \, \pa_\nu v - \frac{t-q}{q} \int_{B_r(x_0)} F_{\mu_+, \mu_-}(v),
\end{equation}
\begin{equation}\label{W e N}
W_{\gamma,t}(v,x_0,r) = \frac{H(v,x_0,r)}{r^{N-1+2\gamma}} \Big( N_t(v,x_0,r)-\gamma \Big)
\end{equation}
whenever the right hand side makes sense. 

Denoting with $\prime$ the derivative with respect to $r$, we have
\begin{equation}\label{der H}
\begin{split}
H'(v,x_0,r) &=  \frac{N-1}{r}H(v,x_0,r) + 2 \int_{S_r(x_0)} v \, \pa_\nu v \\
& =  \frac{N-1}{r}H(v,x_0,r) + 2 D_q(v,x_0,r),
\end{split}
\end{equation}
and hence
\begin{equation}\label{H e W}
\left( \frac{H(v,x_0,r)}{r^{N-1+2\gamma}} \right)' = \frac2r W_{\gamma,q}(v,x_0,r).
\end{equation}
Moreover, proceeding exactly as in \cite[Proposition 2.1]{SoWe} (which concerns the case $\mu_+=\mu_-=1$), it is not difficult to check that
\begin{align*}
\int_{S_r(x_0)} |\nabla u|^2 & = \frac{N-2}{r} \int_{B_r(x_0)} |\nabla v|^2 - \frac{ 2N}{q r} \int_{B_r(x_0)} F_{\mu_+, \mu_-}(v) \\
& \qquad + \int_{S_r(x_0)}\left(2 v_\nu^2 + \frac{2}{q} F_{\mu_+, \mu_-}(v)\right),
\end{align*}
whence
\begin{equation}\label{der D}
\begin{split}
D_t'(v,x_0,r) & = \frac{N-2}{r} D_t(v,x_0,r) - \frac{2N-(N-2)t }{q r} \int_{B_r(x_0)} F_{\mu_+, \mu_-}(v) \\
& + \int_{S_r(x_0)}\left(2 v_\nu^2 + \left(\frac{2-t}{q}\right) F_{\mu_+, \mu_-}(v)\right).
\end{split}
\end{equation}
We derive also the expressions of the derivative $W_{\gamma,t}$.

\begin{proposition}\label{prop: der W}
There holds 
\begin{equation}\label{der W}
\begin{split}
W_{\gamma,t}'(v,x_0,r) & =\frac{2}{r^{N-2+2\gamma}}  \int_{S_r(x_0)} \left( v_\nu-\frac{\gamma}{r} v\right)^2 +  \frac{2-t}{q r^{N-2+2\gamma}} \int_{S_r(x_0)}F_{\mu_+, \mu_-}(v)  \\
&  + \frac{(N-2) t - 2N + 2 \gamma(t-q) }{q r^{N-1+2\gamma}} \int_{B_r(x_0)} F_{\mu_+, \mu_-}(v).
\end{split}
\end{equation}
\end{proposition}
\begin{proof}
In order to simplify the notation, we consider $x_0=0$ and omit the dependence of the functionals with respect to $v$ and $x_0$. By \eqref{D bou}-\eqref{der D}, we have
\[
\begin{split}
W_{\gamma,t}'(r) & = -\frac{2\gamma}{r^{N-1+2\gamma}} D_t(r) - \frac{2N-(N-2)t}{q r^{N-1+2\gamma}}  \int_{B_r} F_{\mu_+, \mu_-}(v)  \\
& + \frac{1}{r^{N-2+2\gamma}} \int_{S_r}\left(2  v_\nu^2 + \left(\frac{2-t}{q}\right) F_{\mu_+, \mu_-}(v)\right) + \frac{2\gamma^2}{r^{N+2\gamma}} H(r)  - \frac{2\gamma}{r^{N-1+2\gamma}} \int_{S_r} v \,\pa_\nu v\\
& = -\frac{4\gamma}{r^{N-1+2\gamma}} \int_{S_r} v \,\pa_\nu v + \frac{(N-2)t-2N+2\gamma(t-q)}{q r^{N-1+2\gamma}}  \int_{B_r} F_{\mu_+, \mu_-}(v) \\
& +  \frac{1}{r^{N-2+2\gamma}} \int_{S_r}\left(2  v_\nu^2 + \left(\frac{2-t}{q}\right)F_{\mu_+, \mu_-}(v)\right) + \frac{2\gamma^2}{r^{N+2\gamma}} \int_{S_r} v^2,
\end{split}
\]
whence the thesis follows.
\end{proof}

We will be particularly interested in the cases $t=q$ and $t=2$. In the latter one, we can easily derive the monotonicity of the Weiss-type functional for a whole range of parameters $\gamma$.

\begin{corollary}\label{cor: W 2 mon}
Let $x_0 \in Z(v)$. If $\gamma \ge 2/(2-q)$, then $W_{\gamma,2}$ is monotone non-decreasing with respect to $r$. Moreover, $W_{\gamma,2}(v,x_0,\cdot)=const.$ for $r_1<r<r_2$ implies that $v$ is $\gamma$-homogeneous with respect to $x_0$ in the annulus $B_{r_2}(x_0) \setminus B_{r_1}(x_0)$.
\end{corollary}

\begin{proof}
The monotonicity follows straightforwardly by Proposition \ref{prop: der W}. If $t=2$, then
\[
W_{\gamma,2}'(r) \ge \frac{2(N-2) - 2N + 2 \gamma(2-q) }{q r^{N-1+2\gamma}} \int_{B_r(x_0)} F_{\mu_+, \mu_-}(v),
\]
and $2(N-2) - 2N + 2 \gamma(2-q) \ge 0$ if and only if $\gamma \ge 2/(2-q)$. 

Now, if $W_{\gamma,2}(v,x_0,\cdot)=const.$ for $r_1<r<r_2$, then $W_{\gamma,2}'(v,x_0,r) = 0$ for any such $r$, and in particular
\[
\nabla v(x) \cdot (x-x_0) - \gamma v(x) = 0 
\]
for almost every $x \in B_{r_2}(x_0) \setminus B_{r_1}(x_0)$.
\end{proof}

For the case $t=q$, we do not have an analogue result, but in any case it is convenient to explicitly observe that
\begin{equation}\label{der W q}
\begin{split}
W_{\gamma,q}'(v,x_0,r) & = \frac{2}{r^{N-2+2\gamma}}\int_{S_r(x_0)} \left( \pa_\nu v-\frac{\gamma}{r} v\right)^2   +  \frac{2-q}{q r^{N-2+2\gamma}} \int_{S_r(x_0)} F_{\mu_+, \mu_-}(v) \\
& - \frac{2N-(N-2)q}{q r^{N-1+2\gamma}} \int_{B_r(x_0)} F_{\mu_+, \mu_-}(v).
\end{split}
\end{equation}
The negative part of $W_{\gamma,q}'$ is
\begin{equation}\label{def phi}
\Phi_\gamma(v,x_0,r):= \frac{2N-(N-2)q}{q r^{N-1+2\gamma}} \int_{B_r(x_0)} F_{\mu_+, \mu_-}(v) \ge 0.
\end{equation}

\section{Strong unique continuation}\label{sec: strong}

In this section we prove Theorem \ref{thm: strong intro}.  The solution $u$ and the point $x_0 \in Z(u)$ will always be fixed, and hence we shall often omit the dependence of the functionals $H$, $D_t$, $W_{\gamma,t}$ and $N_t$ with respect to $u$ and $x_0$. We also set $R:= \dist(x_0,\pa B_1)$.

As an intermediate statement we show that, in the present setting, the classical unique continuation principle holds. 

\begin{proposition}\label{prop: cl ucp}
If $u \equiv 0$ in an open subset of $B_1$, then $u \equiv 0$ in $B_1$.
\end{proposition}
\begin{proof}
When $q \in [1,2)$ with $\lambda_+, \lambda_->0$, this directly follows from the main results in \cite{SoWe} (see also \cite{Ru} for an alternative proof). It remains to discuss the case $\lambda_-=0$. Let then $\lambda_-=0$, and let us consider the open set
\[
U:= \left\{ x \in B_1: \ \text{$u \equiv 0$ in a neighborhood of $x$} \right\}.
\]
Assuming that $U \neq \emptyset$, we show that necessarily $u \equiv 0$ in $B_1$. Since $B_1$ is connected and $U$ is open and non-empty, this follows once we have shown that $\pa U \cap B_1 = \emptyset$. Thus, suppose by contradiction that there exists $x_* \in \pa U \cap B_1$. By continuity, $u(x_*)=0$, and we claim that 
\begin{equation}\label{23 01 1 1}
\text{for every $r>0$ small, $\{u>0\} \cap B_r(x_*) \neq \emptyset$}.
\end{equation} 
Indeed, suppose that there exists $\bar r>0$ such that $u  \le 0$ in $B_{\bar r}(x_*)$; then by \eqref{eq} the function $u$ is harmonic in $B_{\bar r}(x_*)$.
Moreover, since $x_* \in \pa U$, there exists $x \in U \cap B_{\bar r}(x_*)$, and in particular $u \equiv 0$ in $B_{r_x}(x) \cap B_{\bar r}(x^*)$ for some $r_x>0$. As a consequence, the unique continuation principle for harmonic functions yields $u \equiv 0$ in $B_{\bar r}(x_*)$, which ultimately lead to $x_* \in U$, in contradiction with the fact that $U$ is open.

Having proved claim \eqref{23 01 1 1}, we can take $x_1 \in U$ with $|x_1-x_*| < \dist(x_1,\pa B_1)$. We have that \eqref{der H} and \eqref{der D} hold, and furthermore the function
\[
d(r):= \frac{\lambda_+}{q}\int_{B_r(x_1)} (u^+)^q
\]
is non-negative, monotone non-decreasing, and not identically $0$ for $r \in (0,\dist(x_1,\pa B_1))$. These facts allow to proceed exactly as in \cite[Proof of Eq. (2.1)]{SoWe}, obtaining a contradiction. 
 \end{proof}

Now we turn to the proof of Theorem \ref{thm: strong intro}. If point ($i$) of Proposition \ref{thm: blow-up pre} holds, then the result is trivial, and hence we can assume that point ($ii$) holds. We recall that the functional $W_{\gamma,2}$ is monotone non-decreasing in $r$ for $\gamma \ge 2/(2-q)$. Thus, there exists the limit $W_{\gamma,2}(u,x_0,0^+) := \lim_{r \to 0^+} W_{\gamma,2}(u,x_0,r)$. 

\begin{lemma}\label{lem: ex gamma tilde}
There exists $\gamma \ge 2/(2-q)$ sufficiently large that $W_{\gamma,2}(u,x_0,0^+)<0$. Moreover, if $W_{\gamma_1,2}(u,x_0,0^+) < 0 $, then $W_{\gamma_2,2}(u,x_0,0^+) = -\infty$ for every $\gamma_2 > \gamma_1$.
\end{lemma}

\begin{proof}
Let $0<r_1< R$ be such that $H(u,x_0,r_1) \neq 0$; the existence of $r_1$ is ensured by Proposition \ref{prop: cl ucp}, since $u \not \equiv 0$. Then, for $\gamma \ge 2/(2-q)$ sufficiently large, we have
\[
W_{\gamma,2}(r_1) = \frac{1}{r_1^{2 \gamma}} \left[\frac{1}{r_1^{N-2}} D_2(r_1) - \frac{\gamma}{r_1^{N-1}} H(r_1)\right] <0,
\]
and hence by monotonicity (Corollary \ref{cor: W 2 mon}) $W_{\gamma,2}(0^+) \le W_{\gamma,2}(r_1)<0$ for $\gamma$ large. 

Let now $W_{\gamma_1,2}(0^+) < 0 $, and let $\gamma_2>\gamma_1$. Then, for any $0<r< R$,
\[
W_{\gamma_2,2}(r) = \frac{1}{r^{2(\gamma_2- \gamma_1)}} \left[ \frac{1}{r^{N-2+2 \gamma_1}} D_2(r) - \frac{\gamma_2 \pm \gamma_1}{r^{N-1+2 \gamma_1}} H(r) \right] \le \frac{1}{r^{2(\gamma_2-\gamma_1)}}W_{\gamma_1,2}(r),
\]
whence the desired conclusion follows.
\end{proof}
As a corollary:

\begin{corollary}\label{cor: gamma tilde}
Suppose that alternative ($ii$) in Proposition \ref{thm: blow-up pre} holds for $u$ in $x_0$. Then there exists finite
\[
\bar \gamma:= \inf \left\{ \gamma > 0: W_{\gamma,2}(u,x_0,0^+) = -\infty \right\} \in \left[\frac{2}{2-q},+\infty\right).
\]
The limit $W_{\gamma,2}(u,x_0,0^+)$ exists for every $\gamma>0$, and moreover
\[
\begin{cases}   
W_{\gamma,2}(u,x_0,0^+) = 0 & \text{if $0<\gamma< \frac{2}{2-q}$} \\ 
W_{\gamma,2}(u,x_0,0^+) \ge 0 & \text{if $\frac{2}{2-q} \le \gamma < \bar \gamma$} \\
W_{\gamma,2}(u,x_0,0^+) = -\infty & \text{if $\gamma> \bar \gamma$}.
\end{cases}
\]
\end{corollary}
\begin{proof}
The existence of $\bar \gamma \in \R$ follows by Lemma \ref{lem: ex gamma tilde}. Using the fact that alternative ($ii$) of Proposition \ref{thm: blow-up pre} holds for $u$ in $x_0$, it is not difficult to check that $W_{\gamma,2}(0^+) = 0$ for every $\gamma \in (0,2/(2-q))$. Indeed, let us fix any such $\gamma$, and let $\eps>0$ be such that 
\[
\gamma<\frac{2}{2-q}-\eps < \frac{2}{2-q}.
\]
Applying Proposition \ref{thm: blow-up pre} with this $\eps$, we deduce that there exists $C>0$ depending on $\eps$ such that for any small $r>0$
\begin{align*}
| W_{\gamma,2}(r) | & \le  \frac{C}{r^{2 \gamma}}\left( \frac{1}{r^{N-2}} \int_{B_r(x_0)} |\nabla u|^2 + \frac{1}{r^{N-2}} \int_{B_r(x_0)} |u|^q + \frac1{r^{N-1}} \int_{S_r(x_0)} u^2\right) \\
& \le C  r^{2 \left( \frac2{2-q}- \eps -\gamma\right)} \to 0
\end{align*}
as $r \to 0^+$, due to the choice of $\eps$. This proves in particular that $\bar \gamma \ge 2/(2-q)$. In case $\bar \gamma>2/(2-q)$, the existence of a non-negative limit for any $\gamma \in [2/(2-q), \bar \gamma)$ follows directly by Corollary \ref{cor: W 2 mon} and Lemma \ref{lem: ex gamma tilde}.
\end{proof}

We can now prove the validity of the strong unique continuation principle.

\begin{proof}[Proof of Theorem \ref{thm: strong intro}]
We suppose by contradiction that $\cV(u,x_0) = +\infty$ and $u \not \equiv 0$:
\begin{equation}\label{inf ord}
\limsup_{r \to 0^+} \frac{H(u,x_0,r)}{r^{N-1+2\beta}}  = 0 \quad \text{for any $\beta >0$}.
\end{equation}
Clearly this is possible only if Proposition \ref{thm: blow-up pre}-($ii$) holds for $u$ in $x_0$. Let $\bar \gamma$ be defined by Corollary \ref{cor: gamma tilde}, and let us fix $\gamma > \bar \gamma$; we use \eqref{inf ord} with $\beta = 2(\gamma-1)/q$, and deduce that there exist $r_0>0$ small and $C>0$ such that
\begin{equation}\label{st H 9 11 10}
\frac{H(r)}{r^{N-1}} \le C r^{\frac4q(\gamma-1)}  \quad \text{for every $r \in (0,r_0)$}.
\end{equation}
Since $\gamma > \bar \gamma \ge 2/(2-q)$, we have $2(\gamma-1)/q > \gamma$, and hence we deduce that
\begin{equation}\label{st H 9 11 2}
H(r) \le C r^{N-1+2\gamma}  \quad \text{for every $r \in (0,r_0)$}.
\end{equation}
Moreover, always by \eqref{st H 9 11 10}
\begin{equation}\label{st q 9 11}
\begin{split}
\int_{B_r} F_{\lambda_+,\lambda_-}(u) & \le C \int_{B_r} |u|^q = C \int_0^r \left( \int_{S_t} |u|^q \right)dt \\
&\le C\int_0^r H(t)^\frac{q}2 t^{(N-1) \left(1-\frac{q}2\right)}\,dt  \\
&  \le C \int_0^r t^{2(\gamma-1)+N-1}\,dt \le C r^{N-2+2\gamma}
\end{split}
\end{equation}
for every $r \in (0,r_0)$. But then, by \eqref{st H 9 11 2} and \eqref{st q 9 11}, for any $r \in (0,r_0)$ 
\[
W_{\gamma,2}(r) \ge \frac{1}{r^{N-2+2\gamma}} \int_{B_r} |\nabla u|^2 - (1+\gamma) C \ge -(1+\gamma)C,
\]
and in particular $W_{\gamma,2}(0^+) >-\infty$, in contradiction with the fact that, being $\gamma>\bar \gamma$, we have $W_{\gamma,2}(0^+) =-\infty$. 
\end{proof}

\section{Classification of the vanishing order $\cO$, and weak non-degeneracy}\label{sec: strong and non-deg}

In this section we proof two weaker variants of Theorems \ref{thm: very strong V} and \ref{thm: non-deg V}, namely: 

\begin{theorem}\label{thm: very strong}
Let $1 \le q <2$, $\lambda_+,\lambda_->0$, $0 \not \equiv u \in H^1_{\loc}(B_1)$ solve \eqref{eq}, and let $x_0 \in Z(u)$. Then 
\[
\cO(u,x_0) \in \left\{1,\dots, \beta_q, \frac2{2-q} \right\}.
\]
\end{theorem}

\begin{theorem}\label{thm: non-deg}
Let $1 \le q <2$, $\lambda_+,\lambda_->0$, $0 \not \equiv u \in H^1_{\loc}(B_1)$ solve \eqref{eq}, and let $x_0 \in Z(u)$. Then 
\[
\liminf_{r \to 0^+} \frac{\|u\|_{x_0,r}^2}{r^{2 \cO(u,x_0)}}  >0.
\]
\end{theorem}

We recall that $\beta_q$, $\|\cdot\|_{x_0,r}$ and $\cO$ have been defined in \eqref{def beta q}, \eqref{def norm} and Definition \ref{def: order}, respectively.

The proof of Theorems \ref{thm: very strong} and \ref{thm: non-deg} is divided into several intermediate steps. As in the previous section, the dependence of $H$, $D_t$, $W_{\gamma,t}$ and $N_t$ with respect to $u$ and $x_0$ will often be omitted, and $R:= \dist(x_0,\pa B_1)$. 

By Proposition \ref{thm: blow-up pre}, it is not difficult to deduce that 
\[
\mathcal{O}(u,x_0) \in \{1,\dots,\beta_q\} \cup \left[\frac2{2-q},+\infty\right].
\]
If alternative ($i$) in Proposition \ref{thm: blow-up pre} holds for $u$ in $x_0 \in Z(u)$, then by definition there holds $\cO(u,x_0) = d_{x_0} \in \{1,\dots,\beta_q\}$, and hence the thesis of Theorem \ref{thm: very strong} is satisfied. Also the non-degeneracy can be checked directly, using the Taylor expansion.

Therefore, we shall always assume that alternative ($ii$) in Proposition \ref{thm: blow-up pre} holds.

\subsection{Transition exponent for the Weiss-type functional}\label{sub: trans}

Let $\bar \gamma$ be defined in Corollary \ref{cor: gamma tilde}. One of the crucial point in the proof of Theorem \ref{thm: very strong} is the following: 
\begin{proposition}\label{gamma=2/(2-q)}
In case point ($ii$) of Proposition \ref{thm: blow-up pre} holds for $u$ in $x_0$, it results that $\bar \gamma = 2/(2-q)$.
\end{proposition}

The proof proceeds by contradiction: we suppose that $\bar \gamma > 2/(2-q)$, and after several lemmas we will finally obtain a contradiction. 

\begin{lemma}\label{lem: sign D}
If $\bar \gamma > 2/(2-q)$, then $D_2(u,x_0,r) \ge 0$ and $H(u,x_0,r)>0$ for every $r \in (0,R)$.
\end{lemma}

\begin{proof}
Let us take any $\gamma \in (2/(2-q), \bar \gamma)$. Then we know that $W_{\gamma,2}$ is non-decreasing in $r$ (Corollary \ref{cor: W 2 mon}), and that $W_{\gamma,2}(0^+) \ge 0$ (Corollary \ref{cor: gamma tilde}). Therefore, $W_{\gamma,2}(r) \ge 0$ for every $r$, and as a consequence
\begin{align*}
0 \le W_{\gamma,2}(r) = \frac{1}{r^{N-2+2\gamma}} \left[D_2(r) - \frac{\gamma}{r} H(r)  \right] \le \frac{1}{r^{N-2+2\gamma}} D_2(r),
\end{align*}
as claimed. Moreover, we have that $W_{\gamma,q}(r) \ge W_{\gamma,2}(r) \ge 0$ for every $r \in (0,R)$, and hence by \eqref{H e W} we deduce that $H(r)/r^{N-1+2 \gamma}$ is monotone non-decreasing. Assume now by contradiction that $H(r_1)=0$ for some $r_1 \in (0,R)$. By monotonicity, we deduce that $H(r) =0$ for all $r \in (0,r_1)$, and as a consequence $u \equiv 0$ in $B_{r_1}(x_0)$. Therefore, the unique continuation principle proved in \cite{SoWe} implies that $u \equiv 0$ in $B_1$, a contradiction.
\end{proof}

A relevant consequence of the previous lemma is that the frequency functions $N_t(u,x_0,\cdot)$ ($t>0$) are all well defined in $(0,R)$. This fact will be used in what follows. 

Let
\[
\tilde \gamma:= \frac12 \left( \frac2{2-q}+\bar \gamma\right)
\]
be the medium point between $2/(2-q)$ and $\bar \gamma$. We define
\[
\tilde t:= \frac{2N+2 q \tilde \gamma }{N-2+2\tilde \gamma}.
\]
Since $\tilde \gamma>2/(2-q)$ and $q<2$, it results that $\tilde t \in (q,2)$.

\begin{lemma}\label{lem: weiss t}
If $\gamma \ge \tilde \gamma$, then
\[
W_{\gamma, \tilde t}'(u,x_0,r) \ge 0 \quad \text{for every $0<r<R$}.
\]
\end{lemma}

\begin{proof}
By Proposition \ref{prop: der W}, and having observed that $\tilde t  \in (q,2)$, it is sufficient to check that $(N-2) \tilde t-2N + 2 \gamma(\tilde t-q) \ge 0$ for every $\gamma \ge \tilde \gamma$; that is, 
\[
\gamma \ge  \tilde \gamma \quad \implies \quad \gamma \ge \frac{2N-(N-2) \tilde t}{2(\tilde t -q)}.
\]
By definition of $\tilde t$, it is immediate to check that the right hand side coincides with $\tilde \gamma$.
\end{proof}

Now, as in Lemma \ref{lem: ex gamma tilde}
\[
W_{\gamma,\tilde t}(r_1) \le \frac{1}{r_1^{2 \gamma}} \left[\frac{1}{r_1^{N-2}} \int_{B_{r_1}} |\nabla u|^2 - \frac{\gamma}{r_1^{N-1}} H(r_1)\right],
\]
whence we deduce that $W_{\gamma,\tilde t}(0^+)<0$ for $\gamma$ large enough, and proceeding as Corollary \ref{cor: gamma tilde} we define the real number
\[
\bar{ \bar \gamma}:= \inf \left\{ \gamma \ge \tilde \gamma: W_{\gamma,\tilde t}(u,x_0,0^+) = -\infty \right\} \in \left[\tilde \gamma,+\infty\right),
\]
for which
\[
\begin{cases}
W_{\tilde t,\gamma}(u,x_0,0^+) \ge 0 & \text{if $\tilde \gamma \le \gamma <\bar{\bar \gamma}$} \\
W_{\tilde t,\gamma}(u,x_0,0^+) = -\infty & \text{if $\gamma > \bar{\bar \gamma}$}.
\end{cases}
\]

We aim at showing that $\bar{\bar \gamma} = \bar \gamma$, and in this direction we need the following lemma.

\begin{lemma}\label{lem: H infty}
If $\gamma>\bar \gamma$, then 
\[
\limsup_{r \to 0^+} N_2(u,x_0,r) \le \gamma, \quad \text{and} \quad \liminf_{r \to 0^+} \frac{H(u,x_0,r)}{r^{N-1+2\gamma}} = +\infty.
\]
\end{lemma}

\begin{proof}
In this proof we let
\[
a(r):= \frac{H(r)}{r^{N-1+2\gamma}}, \quad \text{and} \quad b(r):= N_2(r) - \gamma.
\]

Since $\gamma> \bar \gamma$, then
\begin{equation}\label{10 11}
-\infty = W_{\gamma,2}(0^+) = \lim_{r \to 0^+} a(r) b(r),
\end{equation}
see \eqref{W e N}. Also, by Lemma \ref{lem: sign D} we have $b(r) \ge -\gamma$, and clearly $a(r) \ge 0$, so that \eqref{10 11} yields
\[
-\gamma \le \liminf_{r \to 0^+} b(r) \le \limsup_{r \to 0^+} b(r) \le 0;
\]
from this, the thesis follows easily.
\end{proof}

\begin{remark}\label{rem: H infty}
For future convenience, we observe that the previous lemma holds true also in case $\bar \gamma= 2/(2-q)$, provided that $H(r)>0$ and $D_2(r) \ge 0$ for any $r>0$ small.
\end{remark}

\begin{lemma}\label{lem: gamma=gamma}
It results that $\bar{\bar \gamma} = \bar \gamma$.
\end{lemma}

\begin{proof}
Since $\tilde t<2$, we have $W_{\gamma,\tilde t}(r) \ge W_{\gamma,2}(r)$ for every $0<r<R$ and $\gamma>0$. Thus $W_{\gamma,\tilde t}(0^+) = -\infty$ implies $W_{\gamma,2}(0^+)=-\infty$ as well, and hence $\bar{\bar \gamma} \ge \bar \gamma$. So let us suppose by contradiction that $\bar{\bar \gamma}> \bar \gamma$, and let us fix $\gamma \in (\bar \gamma, \bar{\bar \gamma})$. We have $W_{\gamma,\tilde t}(0^+) \ge 0$, and $W_{\gamma,\tilde t}$ is monotone non-decreasing in $r$ by Lemma \ref{lem: weiss t}. Therefore, $W_{\gamma,\tilde t}(r) \ge 0$ for $r>0$, and since $q<\tilde t$, we deduce that
\[
W_{q,\gamma}(r) \ge  W_{\tilde t,\gamma}(r)  \ge 0 \quad \text{for every $r \in (0, R)$}.
\]
Recalling equation \eqref{H e W}, it follows that $r \mapsto H(r)/r^{N-1+2\gamma}$ is non-decreasing in $r$, and in particular
\begin{equation}\label{lim exist}
\text{there exists, finite, }\ell:= \lim_{r \to 0^+}  \frac{H(r)}{r^{N-1+2\gamma}} \in [0,+\infty),
\end{equation}
which is in contradiction with Lemma \ref{lem: H infty}.
\end{proof}

As a consequence:

\begin{lemma}\label{lem: N t bou}
If $\gamma> \bar \gamma$, then 
\[
\limsup_{r \to 0^+} N_{\tilde t}(u,x_0,r) \le \gamma.
\]
\end{lemma}

\begin{proof}
Let $\gamma > \bar \gamma$. By Lemma \ref{lem: gamma=gamma}, we have $W_{\gamma, \tilde t}(0^+)=-\infty$, and in particular $W_{\gamma, \tilde t}(r)<0$ for every $r$ small. Also, $H(r)>0$ for every $r \in (0,R)$. Then, recalling \eqref{W e N},
\[
N_{\tilde t}(r)-\gamma  = \frac{ r^{N-1+2\gamma} \, W_{\gamma, \tilde t}(r)}{ H(r)}<0
\]
for every small $r$.
\end{proof}

\begin{lemma}\label{lem: resto 11}
There exists a sequence $0<r_n \to 0^+$ as $n \to \infty$ such that
\[
\frac{1}{r_n^{N-2+2 \bar \gamma}} \int_{B_{r_n}(x_0)} F_{\lambda_+,\lambda_-}(u) \to 0 \quad \text{as $n \to \infty$}.
\]
\end{lemma}

\begin{proof}
Since we are assuming that $\bar \gamma>2/(2-q)$, we can take $\gamma \in [2/(2-q), \bar \gamma)$. Then $W_{\gamma,2}(0^+) \ge 0$ (Corollary \ref{cor: gamma tilde}), and by monotonicity (Corollary \ref{cor: W 2 mon}) this implies that $W_{\gamma,2}(r) \ge 0$ as well, for any $r \in (0,R)$. For any such fixed $r$, we consider the function $\gamma \mapsto W_{\gamma,2}(r)$, and taking the limit as $\gamma \to \bar \gamma^-$ we infer that $W_{\bar \gamma,2}(r) \ge 0$, whence in turn $W_{\bar \gamma,2}(0^+) \ge 0$ follows. Then, for any $\bar r  \in (0,R)$, we have 
\[
0 \le \int_0^{\bar r} W_{\bar \gamma,2}'(s)  \,ds = W_{\bar \gamma,2}(\bar r)- W_{\bar \gamma,2}(0^+) <+\infty.
\]
In particular, recalling Proposition \ref{prop: der W} and using that $\bar \gamma>2/(2-q)$, we deduce that
\begin{equation}\label{10 11 1}
\int_0^{\bar r} \frac1{s}\left(\frac{1}{s^{N-2+2 \bar \gamma}} \int_{B_s(x_0)} F_{\lambda_+,\lambda_-}(u) \right)\,ds <+\infty.
\end{equation}
Since the function $r \mapsto 1/r$ is not integrable in $0$, if 
\[
\liminf_{r \to 0^+} \frac{1}{r^{N-2+2 \bar \gamma}} \int_{B_r(x_0)} F_{\lambda_+,\lambda_-}(u) >0,
\]
then \eqref{10 11 1} would not be possible. This means that the above $\liminf$ has to be $0$, which is the thesis.
\end{proof}

We are finally ready for the:

\begin{proof}[Proof of Proposition \ref{gamma=2/(2-q)}]
Recall that we are assuming by contradiction that $\bar \gamma>2/(2-q)$. Let $\{r_n\}$ be the sequence defined by Lemma \ref{lem: resto 11}, and let
\[
v_n(x):= \frac{u(x_0 + r_n x)}{\left( r_n^{1-N} H(u,x_0,r_n) \right)^\frac12}, \quad \text{defined in} \quad\frac{1}{r_n} B_{R}(x_0) \Supset B_1
\] 
(recall that $H(r_n) >0$ for every $n$, by Lemma \ref{lem: sign D}). By definition
\[
\int_{S_1} v_n^2 = 1, \quad \text{and} \quad \int_{B_1} |\nabla v_n|^2 = \frac{r_n \int_{B_{r_n}(x_0)} |\nabla u|^2}{ \int_{S_{r_n}(x_0)} u^2}.
\]
Now, since $D_2(u,x_0,r) \ge 0$ for every $r  \in (0,R)$ (Lemma \ref{lem: sign D}), we have
\[
\int_{B_{r_n}(x_0)} |\nabla u|^2 \ge \frac2{q} \int_{B_{r_n}(x_0)} F_{\lambda_+,\lambda_-}(u),
\]
and since $\tilde t < 2$ we infer that
\[
\int_{B_{r_n}(x_0)} |\nabla u|^2 \le \frac2{2-\tilde t} \int_{B_{r_n}(x_0)} \left( |\nabla u|^2 - \frac{\tilde t}{q} F_{\lambda_+,\lambda_-}(u) \right) =  \frac2{2-\tilde t} \, D_{\tilde t}(u,x_0,r_n).
\]
As a consequence
\[
\int_{B_1} |\nabla v_n|^2 = \frac{r_n \int_{B_{r_n}(x_0)} |\nabla u|^2}{ \int_{S_{r_n}(x_0)} u^2} \le \frac{2 r_n D_{\tilde t}(u,x_0,r_n)}{(2-\tilde t)H(u,x_0,r_n)} = \frac{2}{2-\tilde t} \, N_{\tilde t}(u,x_0,r_n) \le C
\]
for every $n$ large by Lemma \ref{lem: N t bou}. Therefore, using the compactness of the Sobolev embedding $H^1(B_1) \embed L^q(B_1)$ and of the trace operator $H^1(B_1) \embed L^2(S_1)$, we have that up to a subsequence $v_n \weak v$ weakly in $H^1(B_1)$, strongly in $L^q(B_1)$ and strongly in $L^2(S_1)$. The limit $v$ satisfies
\begin{equation}\label{11 10 4}
\int_{S_1} v^2 =1 \quad \implies \quad v \not \equiv 0 \quad \implies \quad \lim_{n \to \infty} \int_{B_1} F_{\lambda_+,\lambda_-}(v_n) = \int_{B_1} F_{\lambda_+,\lambda_-}(v) > 0.
\end{equation}
On the other hand
\begin{equation}\label{10 11 3}
\begin{split}
\int_{B_1} F_{\lambda_+,\lambda_-}(v_n) &= \frac{1}{r_n^N \left( r_n^{1-N} H(u,x_0,r_n)\right)^\frac{q}2} \int_{B_{r_n}(x_0)} F_{\lambda_+,\lambda_-}(u) \\
&= \frac{r_n^{2 \bar \gamma-2} }{\left( r_n^{1-N} H(u,x_0,r_n)\right)^\frac{q}2} \cdot \frac{1}{r_n^{N-2+2 \bar \gamma}} \int_{B_{r_n}(x_0)} F_{\lambda_+,\lambda_-}(u).
\end{split}
\end{equation}
Since $\bar \gamma >2/(2-q)$, we have $2(2 \bar \gamma-2)/q>2 \bar \gamma$,
so that by Lemma \ref{lem: H infty}
\[ 
\frac{\left( r_n^{1-N} H(r_n)\right)^\frac{q}2}{r_n^{2 \bar \gamma-2} } = \left( \frac{H(r_n)}{r_n^{N-1 + \frac{2}{q}( 2 \bar \gamma-2)}}\right)^\frac{q}{2}\to +\infty.
\]
Therefore, coming back to \eqref{10 11 3} and recalling also Lemma \ref{lem: resto 11}, we infer that
\[
 \int_{B_1} F_{\lambda_+,\lambda_-}(v_n) \to 0,
\]
in contradiction with \eqref{11 10 4}.
\end{proof}

\begin{remark}
The previous argument is valid also in case $\lambda_-=0$, but it does not lead to a conclusive result. Indeed, assuming $\lambda_->0$ we could infer $\int_{B_1} F_{\lambda_+,\lambda_-}(v)>0$ from $v \not \equiv 0$, and this finally gives a contradiction. In case $\lambda_-=0$, we could only deduce that $v \le 0$ a.e. in $B_1$.
\end{remark}

\subsection{Vanishing order and non-degeneracy.}\label{sub: very strong} In this subsection we complete the proofs of Theorems \ref{thm: very strong} and of Theorem \ref{thm: non-deg}. Recall that we are considering the case when alternative ($ii$) in Proposition \ref{thm: blow-up pre} holds. We showed that, in such case, it is well defined the value $\bar \gamma$ given by Corollary \ref{cor: gamma tilde}, and that $\bar \gamma=2/(2-q)$. The main ingredient still missing is the following:

\begin{proposition}\label{prop: non-deg}
It results that
\[
\liminf_{r \to 0^+} \frac{\|u\|_{x_0,r}^2}{r^{2 \bar \gamma}} >0.
\]
\end{proposition}

Suppose by contradiction that for a sequence $0<r_n \to 0^+$ there holds
\begin{equation}\label{absurd non deg}
\lim_{n \to \infty} \frac{\|u\|_{x_0,r_n}^2}{r_n^{2 \bar \gamma}} =0.
\end{equation}
We often let $\|\cdot\|_r:= \|\cdot\|_{x_0,r}$ for the sake of brevity and, for any $r \in (0,R)$, we define
\begin{equation}\label{def: vr non deg}
v_{r}(x):= \frac{u(x_0+rx)}{\|u\|_{x_0,r}} \quad \implies \quad \|v_r\|_{0,1} = 1 \quad \text{for every $r \in (0,R)$}.
\end{equation}
The family $\{v_r: r \in (0,R)\}$ is then bounded in $H^1(B_1)$, hence in $L^q(B_1)$, and in particular
\begin{align*}
\lim_{n \to \infty} W_{\bar \gamma,2}(u,x_0,r_n)   & = \lim_{n \to \infty} \left[ \frac{\|u\|_{r_n}^2}{r_n^{2 \bar \gamma}} \left( \int_{B_1} |\nabla v_{r_n}|^2  - \bar \gamma \int_{S_1} v_{r_n}^2\right) \right. \\
& \hphantom{=\lim_{n \to \infty} \bigg[ \frac{\|u\|_{r_n}^2}{r_n^{2 \bar \gamma}}} 
\left.  - \frac2q \left(\frac{\|u\|_{r_n}^2}{r_n^{2 \bar \gamma}}\right)^{\frac{q}{2}} \int_{B_1} F_{\lambda_+,\lambda_-}(v_{r_n})\right] = 0.
\end{align*}
But the limit $W_{\bar \gamma,2}(u,x_0,0^+)$ exists by monotonicity, see Corollary \ref{cor: W 2 mon}, and hence we proved that $W_{\bar \gamma,2}(u,x_0,0^+) = 0$. This implies, always by monotonicity, that $W_{\bar \gamma,q}(u,x_0,r) \ge W_{\bar \gamma,2}(u,x_0,r) \ge 0$ for every $r \in (0,R)$. Several consequences can be derived by this fact, arguing as in the previous subsection: firstly, as in Lemma \ref{lem: sign D}, we have that $D_2(u,x_0,r) \ge 0$ and $H(u,x_0,r)>0$ for every $r \in (0,R)$. Thus, as observed in Remark \ref{rem: H infty}, we infer that 
\begin{equation}\label{eq: H infty}
\liminf_{r \to 0^+} \frac{H(u,x_0,r)}{r^{N-1+2 \gamma}} = +\infty \quad \text{for every $\gamma > \bar \gamma$}.
\end{equation}
Furthermore, by \eqref{H e W} the function $r \mapsto H(u,x_0,r)/r^{N-1+2 \bar \gamma}$ is monotone non-decreasing, and has a limit as $r \to 0^+$. But
\[
0 < \frac{H(u,x_0,r_n)}{r_n^{N-1+2 \bar \gamma}} \le \frac{\|u\|_{x_0,r_n}^2}{r_n^{2 \bar \gamma}}  \to 0 \quad \text{as $n \to \infty$, by \eqref{absurd non deg}},
\]
and hence
\begin{equation}\label{23 11 1}
\lim_{r \to 0^+} \frac{H(u,x_0,r)}{r^{N-1+2 \bar \gamma}} = 0.
\end{equation}

\begin{lemma}\label{liminf W q su H}
It results that 
\[
\liminf_{r \to 0^+} \frac{r^{N-1+2 \bar \gamma} \, W_{\bar \gamma,q}(u,x_0,r) }{H(u,x_0,r)}=0.
\]
\end{lemma}
\begin{proof}
If the thesis does not hold, then there exist $\eps>0$ and $r_0 \in (0,R)$ such that
\[
\frac{r^{N-1+2 \bar \gamma} W_{\bar \gamma,q}(r) }{H(r)} \ge \eps \quad \text{for every $r \in (0,r_0)$}.
\]
As a consequence, by \eqref{H e W},
\[
\frac{d}{dr}\log\left(\frac{H(r)}{r^{N-1+2\bar \gamma}}\right) = \frac{2}{r}\cdot\frac{r^{N-1+2 \bar \gamma} W_{\bar \gamma,q}(r) }{H(r)} \ge \frac{2\eps}{r}
\]
for $r \in (0,r_0)$, and integrating we deduce that
\[
\frac{H(r)}{r^{N-1+2(\bar \gamma+\eps)}} \le \frac{H(r_0)}{r_0^{N-1+2(\bar \gamma+\eps)}}<+\infty \quad \text{for every $r \in (0,r_0)$}.
\]
In particular, 
\[
\limsup_{r \to 0^+} \frac{H(r)}{r^{N-1+2(\bar \gamma+\eps)}}<+\infty,
\]
in contradiction with \eqref{eq: H infty}.
\end{proof}

\begin{lemma}\label{lim W 2 su H}
It results that 
\[
\lim_{r \to 0^+} \frac{r^{N-1+2 \bar \gamma} \, W_{\bar \gamma,2}(u,x_0,r) }{H(u,x_0,r)}=0.
\]
\end{lemma}

\begin{proof}
In this proof we simply write $B_r$ and $S_r$ instead of $B_r(x_0)$ and $S_r(x_0)$, to ease the notation. Using \eqref{H e W} and \eqref{der W} in case $t=2$ and $\gamma = \bar \gamma = 2/(2-q)$, we compute the derivative
\begin{align*}
\left( \frac{r^{N-1+2 \bar \gamma} \, W_{\bar \gamma,2}(r) }{H(r)} \right)' & = \frac{ \frac{H(r)}{r^{N-1+2 \bar \gamma}} \cdot W_{\bar \gamma,2}'(r) - W_{\bar \gamma,2}(r) \left( \frac{H(r)}{r^{N-1+2 \bar \gamma}}\right)'  }{\left( \frac{H(r)}{r^{N-1+2 \bar \gamma}} \right)^2} \\
& = 
 \frac{\frac{H(r)}{r^{N-1+2 \bar \gamma}} 
\cdot \frac{2}{r^{N-2+2 \bar \gamma}} \int_{S_r} \left( u_\nu- \frac{\bar \gamma}{r}u \right)^2  
 - \frac{2}{r} W_{\bar \gamma,2}(r) W_{\bar \gamma,q}(r)}
 {\left( \frac{H(r)}{r^{N-1+2 \bar \gamma}} \right)^{2}}
\end{align*}
Now, $0 \le W_{\bar \gamma,2}(r) \le W_{\bar \gamma,q}(r)$ for every $r \in (0,R)$, and using also \eqref{D bou} we infer that
\begin{align*}
&\left( \frac{H(r)}{r^{N-1+2 \bar \gamma}} \right)^{2}  \left( \frac{r^{N-1+2 \bar \gamma} \, W_{\bar \gamma,2}(r) }{H(r)} \right)'   \\
& \qquad  \ge \frac{H(r)}{r^{N-1+2 \bar \gamma}} 
\cdot \frac{2}{r^{N-2+2 \bar \gamma}} \int_{S_r} \left( u_\nu- \frac{\bar \gamma}{r}u \right)^2 -  \frac{2}{r} \Big( W_{\bar \gamma,q}(r)\Big)^2 \\
& \qquad = \frac{2}{r^{2N-3+4 \bar \gamma}}  \int_{S_r} \left( u_\nu- \frac{\bar \gamma}{r}u \right)^2  \int_{S_r} u^2 - \frac{2}{r} \left( \frac{1}{r^{N-2+2 \bar \gamma}}\int_{S_r} u u_\nu  - \frac{\bar \gamma}{r^{N-1+2 \bar \gamma}} \int_{S_r} u^2\right)^2 \\
 & \qquad = \frac{2}{r^{2N-3+4 \bar \gamma}} \left[ \int_{S_r} u_\nu^2 \int_{S_r} u^2 - \left(\int_{S_r} u  u_\nu \right)^2\right] \ge 0
\end{align*}
by the Cauchy-Schwarz inequality. As a consequence, recalling that $H(r)>0$ for any $r \in (0,R)$, there exists the limit
\[
\lim_{r \to 0^+} \frac{r^{N-1+2 \bar \gamma} \, W_{\bar \gamma,2}(r) }{H(r)}.
\]
As $0 \le   W_{\bar \gamma,2}(r) \le W_{\bar \gamma,q}(r)$ for every $r \in (0,R)$, the fact that such limit is $0$ follows directly by Lemma \ref{liminf W q su H}.
 \end{proof}

By Lemma \ref{liminf W q su H}, there exists a sequence $0<r_m \to 0^+$ such that
\[
\frac{r_m^{N-1+2 \bar \gamma} \, W_{\bar \gamma,q}(u,x_0,r_m) }{H(u,x_0,r_m)} \to 0
\]
as $m \to \infty$. Notice that this sequence could be different from $\{r_n\}$. In any case, combining Lemmas \ref{liminf W q su H} and \ref{lim W 2 su H}, we deduce that
\begin{equation}\label{u q 0}
 \frac{r_m}{ H(r_m)} \int_{B_{r_m}(x_0)} F_{\lambda_+,\lambda_-}(u) =  \frac{q \, r_m^{N-1+2 \bar \gamma}}{(2-q) H(r_m)} \left( W_{\bar \gamma,q}(r_m) - W_{\bar \gamma,2}(r_m) \right) \to 0
\end{equation}
as $m \to \infty$. Now, let $v_m:= v_{r_m}$ defined by \eqref{def: vr non deg}. Up to a subsequence, $\{v_m\}$ converges weakly in $H^1(B_1)$ and strongly in $L^q(B_1)$ to a limit function $\bar v$.

\begin{lemma}\label{lem: weak limit 0}
It results that $\bar v \equiv 0$ in $B_1$.
\end{lemma}
\begin{proof}
Recall that $\bar \gamma=2/(2-q)$. Due to equation \eqref{u q 0} and the fact that $\|u\|_{r_m}^2 \ge H(r_m)/r_m^{N-1}$, we have 
\begin{align*}
0 \le   \left( \frac{r_m^{N-1+2 \bar \gamma}}{ H(r_m)} \right)^{\frac{2-q}{2}} &\int_{B_1} F_{\lambda_+,\lambda_-}(v_m) 
= \big( r_m^{2 \bar \gamma}\big)^{\frac{2-q}{2}} \frac{r_m^{N-1}}{H(r_m)}  \left(\frac{H(r_m)}{r_m^{N-1}}\right)^{\frac{q}{2}} \int_{B_1} F_{\lambda_+,\lambda_-}(v_m) \\
&  \le  \frac{r_m^{N+1}  \|u\|_{r_m}^q}{ H(r_m)}\int_{B_1}F_{\lambda_+,\lambda_-}(v_m)  = \frac{r_m}{ H(r_m)} \int_{B_{r_m}(x_0)} F_{\lambda_+,\lambda_-}(u) \to 0
\end{align*}
as $m \to \infty$. Since we showed that $H(r_m)/r_m^{N-1+2 \bar \gamma} \to 0$, see \eqref{23 11 1}, we deduce by definition of $F_{\lambda_+,\lambda_-}$ that $v_m \to 0$ strongly in $L^q(B_1)$.
\end{proof}

We are ready for the:

\begin{proof}[Conclusion of the proof of Proposition \ref{prop: non-deg}]
Since $W_{\bar \gamma,q}(r) \ge W_{\bar \gamma,2}(r) \ge 0$ and $\|u\|_{r_m}^2 \ge H(r_m)/r_m^{N-1}$, by Lemmas \ref{liminf W q su H} and \ref{lim W 2 su H} we infer that
\begin{equation}\label{lim W 2 su norma}
\lim_{m \to \infty} \frac{r_m^{2 \bar \gamma} \, W_{\bar \gamma,q}(r_m) }{\|u\|_{r_m}^2} = 0 = \lim_{m \to \infty} \frac{r_m^{2 \bar \gamma} \, W_{\bar \gamma,2}(r_m) }{\|u\|_{r_m}^2},
\end{equation}
whence it follows as above that
\begin{equation}\label{u q 0 2}
0 = \lim_{m \to \infty} \frac{1}{r_m^{N-2}\|u\|_{r_m}^{2}} \int_{B_{r_m}(x_0)} F_{\lambda_+,\lambda_-}(u) = \lim_{m \to \infty} \frac{r_m^2}{\|u\|_{r_m}^{2-q}} \int_{B_1} F_{\lambda_+,\lambda_-}(v_m).
\end{equation}
Moreover, since $v_m \weak 0$ weakly in $H^1(B_1)$ by Lemma \ref{lem: weak limit 0}, the compactness of the trace operator $H^1(B_1) \hookrightarrow L^2(S_1)$ yields 
\begin{equation}\label{trace to 0}
\int_{S_1} v_m^2 \to 0 \quad \text{as $m \to \infty$}. 
\end{equation}
Therefore, collecting \eqref{lim W 2 su norma}-\eqref{trace to 0} we obtain
\begin{align*}
0 & = \lim_{m \to \infty} \frac{r_m^{2 \bar \gamma} \, W_{\bar \gamma,2}(r_m) }{\|u\|_{r_m}^2} \\
& = \lim_{m \to \infty} \left( \int_{B_1} |\nabla v_m|^2 - \frac{2r_m^2}{q \|u\|_{r_m}^{2-q}} \int_{B_1} F_{\lambda_+,\lambda_-}(v_m) - \bar \gamma \int_{S_1} v_m^2 \right) \\
& = \lim_{m \to \infty} \int_{B_1} |\nabla v_m|^2,
\end{align*}
and consequently $\|v_m\|_{0,1}^2 \to 0$ as $m \to \infty$. This contradicts the fact that, by definition, $\|v_m\|_{0,1}^2 = 1$ for every $m$, and completes the proof of Proposition \ref{prop: non-deg}.
\end{proof}

Having established Propositions \ref{gamma=2/(2-q)} and \ref{prop: non-deg}, Theorems \ref{thm: very strong} and \ref{thm: non-deg} follow easily.

\begin{proof}[Proof of Theorems \ref{thm: very strong} and \ref{thm: non-deg}]
We know that, in case alternative ($ii$) in Proposition \ref{thm: blow-up pre} holds, the transition exponent is $\bar \gamma=2/(2-q)$, and 
\[
\liminf_{r \to 0^+} \frac{\|u\|_{x_0,r}^2}{r^{2 \bar \gamma}} >0.
\]
By Definition \ref{def: order} and recalling also Proposition \ref{thm: blow-up pre}, this implies both that $\cO(u,x_0) = \bar \gamma$, and that the non-degeneracy condition is fulfilled.
\end{proof}

We conclude this section with a characterization of the $H^1$-vanishing order in terms of the Weiss-type functionals $W_{\gamma,2}$.

\begin{proposition}\label{prop: o e weiss}
Let $u$ be a solution to \eqref{eq}, and let $x_0 \in Z(u)$. The value $\cO(u,x_0)$ is characterized by 
\[
\cO(u,x_0) = \inf \left\{ \gamma>0: \lim_{r \to 0^+} W_{\gamma,2}(u,x_0,r)=-\infty \right\}.
\]
Moreover
\[
\lim_{r \to 0^+} W_{\gamma,2}(u,x_0,r) = \begin{cases} 0 & \text{if $0<\gamma < \cO(u,x_0)$} \\
-\infty & \text{if $\gamma > \cO(u,x_0)$}. \end{cases}
\]
\end{proposition}

\begin{remark}
The proposition is saying that the vanishing order of $u$ in $x_0$ coincides with the exponent $\bar \gamma$ defining the transition in the limits of the Weiss-type functional. 
\end{remark}
\begin{proof}
Assume at first that case ($i$) in Proposition \ref{thm: blow-up pre} holds. By Definition \ref{def: order}, it follows straightforwardly that $\cO(u,x_0) = d_{x_0}$,. Moreover, since $d_{x_0}<2/(2-q)$, we have
\[
\frac{1}{r^{N-2}} \int_{B_r(x_0)} |\nabla u|^2 = \frac{1}{r^{N-2}} \int_{B_r} |\nabla P_{x_0}|^2 + o(r^{2d_{x_0}}) = r^{2d_{x_0}} \int_{B_1} |\nabla P_{x_0}|^2 + o(r^{2d_{x_0}}),
\]
\[
\frac{1}{r^{N-2}} \int_{B_r(x_0)} F_{\lambda_+,\lambda_-}(u) \le C r^{q d_{x_0} + 2} = o(r^{2d_{x_0}}),
\]
and
\[
\frac1{r^{N-1}} \int_{S_r(x_0)} u^2 = \frac1{r^{N-1}} \int_{S_r(x_0)} P_{x_0}^2 + o(r^{2d_{x_0}}) = r^{2 d_{x_0}} \int_{S_1} P_{x_0}^2 + o(r^{2d_{x_0}})
\]
as $r \to 0^+$. Therefore, there exists $C>0$ such that
\begin{align*}
W_{\gamma,2}(u,x_0,r) &= \frac{H(u,x_0,r)}{r^{N-1+2\gamma}} \Big( N_2(u,x_0,r)-\gamma \Big) \\
& = \frac1{r^{2 \gamma}}\left( C r^{2d_{x_0}} + o(r^{2d_{x_0}}) \right) \left( \frac{\int_{B_1} |\nabla P_{x_0}|^2}{\int_{S_1} P_{x_0}^2} -\gamma + o(1) \right) \\
& = \frac1{r^{2\gamma}}\left( C r^{2d_{x_0}} + o(r^{2d_{x_0}}) \right) \left( d_{x_0} - \gamma +o(1) \right)
\end{align*}
as $r \to 0^+$, where we used \eqref{W e N}, the expansion given by Proposition \ref{thm: blow-up pre}, and the well known fact that the Almgren frequency of a homogeneous harmonic polynomial coincides with the degree of homogeneity. 

It is now immediate to deduce that if $0<\gamma \le d_{x_0}$, then $W_{\gamma,2}(u,x_0,0^+) =0$, while if $\gamma > d_{x_0}$, then $W_{\gamma,2}(0^+) = -\infty$. That is, the transition exponent for the Weiss-type functional $W_{\gamma,2}$ coincides with $d_{x_0} = \cO(u,x_0)$, as claimed.

Let us suppose now that ($ii$) in Proposition \ref{thm: blow-up pre} holds. Then we proved that $\cO(u,x_0) = 2/(2-q)$, and in doing this we used the fact that the transition exponent is $\bar \gamma=2/(2-q)$ as well, see Proposition \ref{gamma=2/(2-q)}.
\end{proof}

\section{Upper semi-continuity of the vanishing order map}\label{sec: upper}

A key ingredient of the proofs of both Theorems \ref{thm: blow-up} and \ref{thm: struct sing} is the following upper semi-continuity property for the $H^1$-vanishing order map $x_0 \in Z(u) \mapsto \cO(u,x_0)$.

\begin{proposition}\label{prop: upper}
Let $1 \le q<2$, $\lambda_+, \lambda_->0$, $\{\mu_k\}$ be a bounded sequence of positive numbers, and let $\{v_k\} \subset H^1_{\loc}(B_3)$ be such that
\begin{equation}\label{eq upper}
-\Delta v_k = \mu_k\Big( \lambda_+ (v_k^+)^{q-1} -\lambda_- (v_k^-)^{q-1} \Big) \quad \text{in $B_3$}.
\end{equation}
Suppose that $p_k \in Z(v_k) \cap B_1$, and that $p_k \to \xi$, $\mu_k \to \mu \ge 0$ and $v_k \to \varphi$ in $W^{1,\infty}_{\loc}(B_3)$ as $k \to \infty$. Then
\[
\mathcal{O}(\varphi,\xi) \ge \limsup_{k \to \infty} \mathcal{O}(v_k,p_k).
\]
\end{proposition}

\begin{remark}
In case $\mu=0$, the limit function $\varphi$ is harmonic. The $H^1$-vanishing order of $\varphi$ in one of its zeros $\xi$ can defined exactly as in Definition \ref{def: order}. Since for each harmonic function an expansion of the type of Proposition \ref{thm: blow-up pre}-($i$) holds with $d_{x_0} \in \N$ (possibly larger than $\beta_q$), proceeding as in the proof of Proposition \ref{prop: o e weiss} it is easy to deduce that $\cO(\varphi,\xi)=d_{x_0}$.
\end{remark}

It is well known that statements like Proposition \ref{prop: upper} are extremely useful for the study of the nodal set of solutions to elliptic equations, see \cite{Han94, Han00}. For linear (or superlinear) equations, the proof of Proposition \ref{prop: upper} follows by Almgren's monotonicity formula, which is not valid for sublinear equations. An alternative approach consists in using improved Schauder estimate, as in \cite{Han00}, but once again such a method seems not easily extendable due to the sublinear character of equation \eqref{eq upper}. We develop therefore an ad-hoc iterative argument based on the study of the Weiss-ype functionals $W_{\gamma,q}$. We recall the expression of \eqref{der W q}, and in particular the definition of the negative part of $W_{\gamma,q}'$, equation \eqref{def phi}.

We start with a few auxiliary statements. 

\begin{lemma}\label{lem 2}
Let $v \in H^1(B_1(x_0))$ be a solution to \eqref{eq mu} with $\mu_+ = \mu \lambda_+$ and $\mu_- = \mu \lambda_-$, for some $\mu>0$. Suppose that there exist $1 \le \sigma < 2/(2-q)$ and $\bar C>0$ such that 
\[
H(v,x_0,r) \le \bar C r^{N-1+2\sigma} \quad \text{for every $r \in (0,1)$}.
\]
Then $\Phi_\gamma(v,x_0,s) \in L^1(0,1)$ for every $\gamma \in \left[\sigma, \frac{2+\sigma q}{2} \right)$, and more precisely there exists a constant $C=C(N,q,\lambda_+,\lambda_-)>0$ such that for every $\gamma \in \left[\sigma, \frac{2+\sigma q}{2} \right)$ and every $r \in (0,1)$
\[
\int_0^r \Phi_\gamma(v,x_0,s)\,ds \le \frac{C \mu {\bar{C}}^{\frac{q}2} }{2+\sigma q -2 \gamma} r^{ 2+\sigma q -2 \gamma}.
\]
\end{lemma}

\begin{proof}
For every $r \in (0,1)$
\begin{align*}
\int_{B_r(x_0)} F_{\mu \lambda_+,\mu \lambda_-}(v)& \le \mu \max\{ \lambda_+,\lambda_-\} \int_0^r  \left( \int_{S_t(x_0)} |v|^q \right) \, dt \\
& \le C \mu \int_0^r \left( H(v,x_0,t) t^{1-N}\right)^\frac{q}2 t^{N-1}\,dt  \\
&\le C \mu \bar C^{\frac{q}2} \int_0^r t^{\sigma q+N-1}\,dt   \le C \mu \bar C^{\frac{q}2} r^{\sigma q + N}.
\end{align*}
As a consequence, recalling definition \eqref{def phi} of $\Phi_\gamma$, we obtain the desired result.
\end{proof}

\begin{lemma}\label{lem 3}
Let $v \in H^1(B_1(x_0))$ be a solution to \eqref{eq mu} with $\mu_+ = \mu \lambda_+$ and $\mu_- = \mu \lambda_-$, for some $\mu>0$. 
Suppose that, for some $1 \le \gamma < 2/(2-q)$ and $\bar C, p>0$, there holds: 
\begin{itemize}
\item[($i$)] $W_{\gamma,q}(v,x_0,0^+) = 0$; 
\item[($ii$)] $\int_0^r \Phi_\gamma(v,x_0,s)\,ds \le \bar C r^p$ for every $r \in (0,1)$.
\end{itemize}
Then there exist $\tilde C>0$ depending on $N$, $q$, $\bar C$, $\lambda_+$, $\lambda_-$ and on upper bounds on $\|u\|_{H^1(B_1(x_0))}$ and on $\mu$, such that for every $\eps>0$ and $r \in (0,1)$
\[
\frac{H(v,x_0,r)}{r^{N-1+2(\gamma-\eps)}} \le \tilde C r^{2\eps} |\log r|.
\]
\end{lemma}

\begin{proof}
At first, we observe that
\begin{align*}
|W_{\gamma,q}(v,x_0,1) | &\le C \left(\|\nabla u\|^2_{L^2(B_1(x_0))} +  \|u\|^q_{L^q(B_1(x_0))}  +\frac{2}{2-q} \|u\|^2_{L^2(S_1(x_0))} \right) \\
& \le C \left(1+ \|u\|_{H^1(B_1(x_0))}^2\right)=:C_1
\end{align*}
with $C_1$ having the same dependence as $\tilde C$ in the thesis. Therefore, using assumptions ($i$) and ($ii$) and recalling that $\Phi_\gamma$ is the negative part of $W_{\gamma,q}'$, we deduce that
\begin{align*}
0& \le \int_0^1 \big( W_{\gamma,q}'(v,x_0,s) \big)^+ \,ds  \\
& = W_{\gamma,q}(v,x_0,1)- W_{\gamma,q}(v,x_0,0^+) + \int_0^1 \Phi_\gamma(v,x_0,s)\,ds \\
& \le C_1 + \bar C =: C_2.
\end{align*}
By \eqref{der W}, this implies in particular that
\[
2\int_0^1 \frac{1}{r^{N-2+2\gamma}} \left(\int_{S_r(x_0)} \left( \pa_\nu v -\frac{\gamma}{r} v\right)^2\,d\sigma \right)\, dr\le C_2.
\]
Changing variable in the integral, and introducing
\[
g(r,\theta):= \nabla v(x_0+r\theta) \cdot \theta - \frac{\gamma}{r} v(x_0+r \theta),
\]
this inequality can be re-written as
\begin{equation}\label{stima g}
\int_0^1 \frac{2}{r^{2\gamma-1}} \left(\int_{S_1} g^2(r,\theta) \, d\theta\right)\, dr \le C_2.
\end{equation}
Now, let $w(r,\theta):= v(x_0+r\theta)$. We have
\[
\frac{\partial}{\pa r}  \frac{w(r,\theta)}{r^\gamma} = \frac{g(r,\theta)}{r^\gamma},
\]
so that
\begin{align*}
\left| \frac{w(r,\theta)}{r^\gamma}\right| & = \left| w(1,\theta) - \int_r^1  \frac{g(s,\theta)}{s^{\gamma}}\,ds\right| \\
& \le | w(1,\theta)| + \left( \int_r^1 \frac{g^2(s,\theta)}{s^{2\gamma-1}}\,ds\right)^\frac12 \left( \int_r^1 \frac{1}{s}\,ds\right)^\frac12 \\
&  \le | w(1,\theta)| + \left( \int_r^1 \frac{g^2(s,\theta)}{s^{2\gamma-1}}\,ds\right)^\frac12 |\log r|^\frac12.
\end{align*}
Taking the square of both sides and integrating in $\theta$, we obtain by Fubini-Tonelli's theorem that
\begin{align*}
\frac{1}{r^{N-1+2\gamma}} \int_{S_r(x_0)} v^2 d\sigma &= \frac{1}{r^{2\gamma}} \int_{S_1} w^2(r,\theta)\,d\theta \\
& \le 2  \|v\|_{L^2(S_1(x_0))} + 2 |\log r| \int_{S_1}\left( \int_r^1 \frac{g^2(s,\theta)}{s^{2\gamma-1}}\,ds\right)\,d\theta \\
& \le C \|v\|_{H^1(B_1(x_0))} + 2 |\log r|  \int_0^1 \left( \int_{S_1} \frac{g^2(s,\theta)}{s^{2\gamma-1}}\,d\theta\right)\,ds \\
& \le \left(C \|v\|_{H^1(B_1(x_0))} + C_2\right) |\log r|
\end{align*}
for any $r \in (0,1)$, where we used estimate \eqref{stima g}. Multiplying the first and the last term by $r^{2\eps}$, the thesis follows.
\end{proof}

\begin{lemma}\label{lem: seq exp}
Let 
\begin{equation}
\label{seq exp}
\begin{cases}
\sigma_0=1 \\
\sigma_k=\frac12\left(\frac{2+q \sigma_{k-1}}{2}\right) + \frac{\sigma_{k-1}}2 & \text{if }k \ge 1.
\end{cases}
\end{equation}
The sequence $\{\sigma_k\}$ is monotone increasing and converges to $2/(2-q)$ as $k \to \infty$. Moreover, $\sigma_k < (2+q \sigma_{k-1})/2$ for every $k$.
\end{lemma}
\begin{proof}
At first we claim that $\sigma_k <2/(2-q)$ for every $k$. Clearly $\sigma_0=1<2/(2-q)$; if $\sigma_k<2/(2-q)$, it is not difficult to check that $\sigma_{k+1}<2/(2-q)$ as well, so that the claim follows by the induction principle.

Now we show that $\sigma_k>\sigma_{k-1}$ for every $k$. Simple computations show that this inequality is satisfied provided that $\sigma_{k-1} < 2/(2-q)$, which is always the case, as proved previously.

By monotonicity there exists the limit $\bar \sigma =\lim_k \sigma_k$. By definition of $\sigma_k$, this limit satisfies
\[
\frac12\left(\frac{2+q \bar \sigma}{2}\right) + \frac{\bar \sigma}2 = \bar \sigma \quad \iff \quad \bar \sigma = \frac{2}{2-q}.
\]

It remains to prove that $\sigma_k < (2+q \sigma_{k-1})/2$ for every $k$. Once again, simple computations show that this inequality holds for all $k$ such that $\sigma_{k-1}<2/(2-q)$, that is, for every $k$.
\end{proof}

Finally:

\begin{lemma}\label{lem: char O}
Let $1 \le q <2$, $v$ be a solution to \eqref{eq mu} for some $\mu>0$, and let $x_0 \in Z(v)$. Then the limit $W_{\gamma,q}(v,x_0,0^+)$ exists for every $\gamma \in (0,2/(2-q))$, and
\[
\begin{cases}   
W_{\gamma,q}(v,x_0,0^+) = 0 & \text{if $0<\gamma< \cO(v,x_0)$} \\ 
W_{\gamma,q}(v,x_0,0^+) = -\infty & \text{if }\cO(v,x_0)<\gamma<\frac{2}{2-q}.
\end{cases}
\]
Moreover, if $\cO(v,x_0)<2/(2-q)$, then $W_{\cO(v,x_0),q}(v,x_0,0^+) = 0$.

If $v$ is a harmonic function, then $\cO(v,x_0) \in \N$ and
\[
\begin{cases}   
W_{\gamma,q}(v,x_0,0^+) = 0 & \text{if $0<\gamma< \cO(v,x_0)$} \\ 
W_{\gamma,q}(v,x_0,0^+) = -\infty & \text{if }\cO(v,x_0)<\gamma.
\end{cases}
\]
\end{lemma}
The proof is essentially the same as the one of Proposition \ref{prop: o e weiss}, see also Corollary \ref{cor: gamma tilde}, and hence is omitted.

We are ready to proceed with the:

\begin{proof}[Proof of Proposition \ref{prop: upper}]
In Theorem \ref{thm: very strong} we proved that, when $\mu=1$, the $H^1$-vanishing order $\cO(u,x_0)$ of a solution to \eqref{eq mu} can only take a value in $\{1,\dots,\beta_q,2/(2-q)\}$. The result trivially generalizes to any $\mu>0$ by scaling. In particular, since the generalized vanishing order can take only a finite number of values, it is not restrictive to suppose that $\mathcal{O}(v_k,p_k) = d$ for every $k$.
 
 \medskip
 
If $d=1$ there is nothing to prove, since by convergence $\varphi(\xi) = 0$, and hence $\mathcal{O}(\varphi,\xi) \ge 1$. 

\medskip

Let us now suppose that $1<d \le \beta_q$, defined by \eqref{def beta q}. By assumption, $\{v_k\}$ is a sequence of equi-Lipschitz continuous function. We denote by $L$ the uniform Lipschitz constant. Then, since $v_k(p_k)=0$ for every $k$, we have
\[
H(v_k,p_k,r) \le L^2 \int_{S_r(p_k)} |x-p_k|^2 \, d \sigma \le C L^2 r^{N+1} =: C_0 r^{N+1},
\]
for every $r \in (0,1)$ and $k \in \N$. Notice that $C_0$ is independent of $k$. 

At this point we recall the definition of $\sigma_0$ and $\sigma_1$ from \eqref{seq exp}, and apply Lemma \ref{lem 2} with $\sigma=\sigma_0=1$ and $\gamma= \sigma_1 + \delta_1$, where $\delta_1>0$ is chosen in such a way that $\sigma_1+\delta_1 < (2+\sigma_0q)/2$: we infer that
\begin{equation}\label{hp 2 lem 3 1}
\int_0^r \Phi_{\sigma_1+\delta_1}(v_k,p_k,s)\,ds \le \frac{\bar C_0}{2+\sigma_0q -2(\sigma_1+\delta_1)} r^{2+\sigma_0q -2(\sigma_1+\delta_1)},
\end{equation}
for a positive constant $\bar C_0$ independent of $k$. Notice that $(2+\sigma_0q)/2 <2 \le 2/(2-q)$, and hence 
\begin{equation}\label{hp 1 lem 3 1}
W_{\sigma_1+\delta_1}(v_k,p_k,0^+) =0
\end{equation}
by Lemma \ref{lem: char O}. Equations \eqref{hp 2 lem 3 1} and \eqref{hp 1 lem 3 1} enable us to apply Lemma \ref{lem 3} with $\gamma=\sigma_1+\delta_1$ and $\eps=\delta_1$, deducing that there exists $C_1>0$ depending on $N$, $q$ and on $\sup_k\{\|v_k\|_{H^{1}(B_{1}(\xi))}\}$ (in particular, $C_1$ is independent of $k$) such that
\[
H(v_k,p_k,r) \le C_1 r^{N-1+2\sigma_1} r^{2 \delta_1 } |\log r| \le C_1 r^{N-1+2\sigma_1}.
\]

The previous argument can be iterated as follows: first, we observe that either $(2+\sigma_1 q)/2 > d$, or not. 

\emph{Case 1) $(2+\sigma_1 q)/2 > d$.} We apply Lemma \ref{lem 2} with $\sigma=\sigma_1$ and $\gamma = d$, deducing that
\begin{equation}\label{concl}
\int_0^r \Phi_{d}(v_k,p_k,s)\,ds \le \frac{\bar C_1}{2+\sigma_1 q -2d} r^{2+\sigma_1 q -2d}
\end{equation}
for a positive constant $\bar C_1$ independent of $k$. Now, recalling the definition of $\Phi_\gamma$ (see \eqref{def phi}), we have $W_{d,q}'(v_k,p_k,r) \ge -\Phi_d(v_k,p_k,r)$, whence
\begin{align*}
W_{d,q}(v_k,p_k,r) & \ge W_{d,q}(v_k,p_k,0^+) -  \int_{0}^r  \Phi_d(v_k,p_k,s)\,ds \ge - C r^{2+\sigma_1 q -2d}
\end{align*}
where the last equality follows by Lemma \ref{lem: char O} and \eqref{concl}. Here $C$ denotes a a positive constant $C$ independent of $k$, and the inequality holds for any $r \in (0,1)$. Now we pass to the limit in $k$: by $W^{1,\infty}$ convergence, we deduce that
\[
W_{d,q}(\varphi,\xi,r) \ge  - C r^{2+\sigma_1 q -2d},
\]
and taking the limit as $r \to 0^+$, we finally obtain $W_{d,q}(\varphi,\xi,0^+) \ge 0$. By Lemma \ref{lem: char O}, it follows that $\mathcal{O}(\varphi,\xi) \ge d$, as desired.

\emph{Case 2) $(2+\sigma_1 q)/2 \le d$.} In this case we apply Lemma \ref{lem 2} with $\sigma=\sigma_1$ and $\gamma = \sigma_2+\delta_2$, $\sigma_2$ defined by \eqref{seq exp}, and $\delta_2>0$ small so that $\sigma_2+\delta_2 < (2+\sigma_1 q)/2$. We deduce that
\begin{equation}\label{hp 2 lem 3 2}
\int_0^r \Phi_{\sigma_2+\delta_2}(v_k,p_k,s)\,ds \le \frac{\bar C_1}{2+\sigma_1 q -2(\sigma_2+\delta_2)} r^{2+\sigma_1 q -2(\sigma_2+\delta_2)}
\end{equation}
for $\bar C_1>0$ independent of $k$. Moreover, being $\sigma_2+\delta_2 < d$, we have that 
\begin{equation}\label{hp 1 lem 3 2}
W_{\sigma_2+\delta_2,q}(v_k,p_k,0^+) = 0
\end{equation}
by Lemma \ref{lem: char O}. Equations \eqref{hp 2 lem 3 2} and \eqref{hp 1 lem 3 2} enables us to apply Lemma \ref{lem 3} with $\gamma=\sigma_2+\delta_2$ and $\eps=\delta_2$, deducing that there exists $C_2>0$ such that
\[
H(v_k,p_k,r) \le C_2 r^{N-1+2\sigma_2} r^{2\delta_2} |\log r| \le C_2 r^{N-1+2\sigma_2}.
\]
At this point we check whether $(2+q \sigma_2)/2 >d$ or not. If yes, we follow Case 1 to deduce that $\mathcal{O}(\varphi,\xi) \ge d$. If not, we iterate the previous argument once again obtaining
\[
H(v_k,p_k,r) \le C_3 r^{N-1+2\sigma_3},
\]
and so on. By Lemma \ref{lem: seq exp}, we are sure that there exists $k \in \N$ such that $(2+\sigma_k q)/2> d$, so that the proof is complete after a finite number of iterations.

\medskip

The case $d= 2/(2-q)$ can be treated similarly to the previous one. We use the same iteration above, deducing after a finite number of steps that $W_{\beta_q+\eps,q}(\varphi, \xi, 0^+) = 0$ for some positive $\eps$. Notice that $\varphi$ solves \eqref{eq mu} for some $\mu \ge 0$. If $\mu>0$, it follows then by Theorem \ref{thm: very strong} and the first part of Lemma \ref{lem: char O} that $\cO(\varphi,\xi) = 2/(2-q)$. If $\mu=0$, always by Lemma \ref{lem: char O}, we infer that $\cO(\varphi,\xi)$ must be an integer, larger than $\beta_q+\eps$, thus $\cO(\varphi,\xi) \ge 2/(2-q)$. In both cases, the proof is complete.
\end{proof}

\section{Blow-up limits}\label{sec: blow-up}

The purpose of this section consists in showing the following result:

\begin{theorem}\label{thm: blow-up '}
Let $1 \le q<2$, $\lambda_+,\lambda_->0$, $0 \not \equiv u \in H^1_{\loc}(B_1)$ solve \eqref{eq}, and $x_0 \in Z(u)$. If $\cO(u,x_0) = 2/(2-q)$, then for every sequence $0<r_n \to 0^+$ we have, up to a subsequence,
\[
\frac{u(x_0+r_n x)}{\|u\|_{x_0,r_n}}
 \to \bar u \qquad \text{in $C^{1,\alpha}_{\loc}(\R^N)$ for every $0<\alpha<1$},
\]
where $\bar u$ is a $2/(2-q)$-homogeneous non-trivial solution to 
\begin{equation}\label{eq limit 1 12}
-\Delta \bar u= \mu \left( \lambda_+ (\bar u)^{q-1} -\lambda_- (\bar u)^{q-1} \right) \quad \text{in $\R^N$}
\end{equation}
for some $\mu \ge 0$. Moreover, the case $\mu = 0$ is possible only if $2/(2-q) \in \N$. 
\end{theorem}

Some preliminary lemmas are needed. Throughout this section the value $2/(2-q)$ will be denoted by $\gamma_q$. For $0<r<R <\dist(x_0,\pa B_1)$, we consider
\[
v_r(x):= \frac{u(x_0+rx)}{\|u\|_{x_0,r}} \quad \implies \quad \|v_r\|_{0,1} = 1.
\]
Then 
\[
-\Delta v_r = 
 \left( \frac{r^{\gamma_q}}{\|u\|_{x_0,r}} \right)^{\frac{2}{\gamma_q}} \Big( \lambda_+ (v_r^+)^{q-1} - \lambda_- (v_r^-)^{q-1} \Big)\quad \text{in } \frac{1}{r}B_R.
\]
Notice that the scaled domains exhaust $\R^N$ as $r \to 0^+$, and that there exists $C>0$ such that
\[
0 < \alpha_r := \left( \frac{r^{\gamma_q}}{\|u\|_{x_0,r}} \right)^{\frac{2}{\gamma_q}} \le C \qquad \text{for every $0<r<R$},
\]
by non-degeneracy, Theorem \ref{thm: non-deg}.

\begin{lemma}\label{lem: weak to strong}
For $\rho>0$, let $0<r_n \to 0^+$ be such that $v_{r_n} \to v$ weakly in $H^1(B_\rho)$ and strongly in $L^2(B_\rho)$. Then $v_{r_n} \to v$ strongly in $H^1_{\loc}(B_\rho)$.
\end{lemma}

\begin{proof}
The result follows easily testing the equation for $v_{r_n}$ against $(v_{r_n}-v)\eta$, where $\eta$ is an arbitrary cut-off function in $C^\infty_c(B_\rho)$, and passing to the limit. Indeed, we have
\begin{align*}
\int_{B_\rho} \eta \nabla v_{r_n} \cdot \nabla (v_{r_n}-v) & = - \int_{B_\rho} (v_{r_n}-v) \nabla v_{r_n} \cdot \nabla \eta \\
& + \alpha_{r_n} \int_{B_{\rho}} \Big( \lambda_+ (v_{r_n}^+)^{q-1} -\lambda_-(v_{r_n}^-)^{q-1} \Big)(v_{r_n}-v) \eta,
\end{align*}
and by our assumptions the right hand side tends to $0$ as $n \to \infty$ (recall that $\{\alpha_{r_n}\}$ is bounded). Regarding the left hand side, we have by weak convergence
\[
\int_{B_\rho} \eta \nabla v_{r_n} \cdot \nabla (v_{r_n}-v) = \int_{B_{\rho}} \eta \left(|\nabla v_{r_n}|^2 -|\nabla v|^2\right) +o(1).
\]
Therefore
\[
\lim_{n \to \infty} \int_{B_{\rho}} \eta \left(|\nabla v_{r_n}|^2 -|\nabla v|^2\right) = 0
\]
for every $\eta \in C^\infty_c(B_\rho)$, whence $\|v_{r_n}\|_{H^1(K)} \to \|v\|_{H^1(K)}$ as $n \to \infty$, for any compact set $K \Subset B_\rho$. In turn, the thesis follows.
\end{proof}

\begin{lemma}\label{lem: v_r bdd}
Let $\rho>1$ be fixed. There exists $ r_\rho>0$ small enough such that the family $\{v_r: r \in (0,r_\rho)\}$ is bounded in $H^1(B_\rho)$.
\end{lemma}
\begin{proof}
It results that $\|v_r\|_{0,\rho} = \|u\|_{x_0,\rho r}/\|u\|_{x_0,r}$, and hence the thesis follows if there exist $r_\rho>0$ and a constant $C_\rho>0$ such that
\begin{equation}\label{1 12 1}
\frac{\|u\|_{x_0,\rho r}}{ \|u\|_{x_0,r}} \le C_\rho, \quad \text{for every $0<r<r_\rho$}.
\end{equation}
Let us suppose by contradiction that for a sequence $0<r_n \to 0^+$ it results 
\begin{equation}\label{30 11 1}
\frac{\|u\|_{x_0,\rho r_n}}{ \|u\|_{x_0,r_n}} \to +\infty \quad \text{as $n \to \infty$}.
\end{equation}
We claim that in such case
\begin{equation}\label{30 11 2}
\frac{\|u\|_{x_0,\rho r_n}}{ (\rho r_n)^{\gamma_q}} \to +\infty \quad \text{as $n \to \infty$}.
\end{equation}
If not, by non-degeneracy (Theorem \ref{thm: non-deg}), up to a subsequence we would have that 
\[
\|u\|_{x_0,\rho r_n} \le C (\rho r_n)^{\gamma_q} \le C \rho^{\gamma_q} \|u\|_{x_0,r_n},
\]
against \eqref{30 11 1}. Thus \eqref{30 11 2} holds. Now, by the Poincar\'e inequality
\begin{align*}
\frac1{r^{N-2}} \int_{B_r(x_0)} F_{\lambda_+,\lambda_-}(u) & \le \frac{C}{r^{N-2}} \int_{B_r(x_0)} |u|^q \le 
C r^2 \left( \frac1{r^N} \int_{B_r(x_0)} u^2 \right)^{\frac{q}{2}}\\
& \le Cr^2 \|u\|_{x_0,r}^q = C \left( \frac{r^{\gamma_q}}{\|u\|_{x_0,r}} \right)^\frac{2}{\gamma_q} \|u\|_{x_0,r}^2
\end{align*}
for every $r>0$, and hence \eqref{30 11 2} implies that
\[
\frac{1}{\|u\|_{x_0,\rho r_n}^2}\cdot \frac1{(\rho r_n)^{N-2}} \int_{B_{\rho r_n}(x_0)} F_{\lambda_+,\lambda_-}(u) \to 0 \qquad \text{as $n \to \infty$}.
\]
This estimate and the monotonicity of $W_{\gamma_q,2}$ (Corollary \ref{cor: W 2 mon}) yield
\begin{align*}
C & \ge W_{\gamma_q,2}(u,x_0,\rho r_n)  \\
& = \frac{\|u\|_{x_0, \rho r_n}^2}{(\rho r_n)^{2 \gamma_q}}\left[1 - \frac{(\gamma_q+1) H(u,x_0, \rho r_n)}{(\rho r_n)^{N-1}\|u\|_{x_0, \rho r_n}^2 } -  \frac2{q(\rho r_n)^{N-2} \|u\|_{x_0,\rho r_n}^2} \int_{B_{\rho r_n}(x_0)} F_{\lambda_+,\lambda_-}(u)\right] \\
& \ge  \frac{\|u\|_{x_0, \rho r_n}^2}{(\rho r_n)^{2 \gamma_q}} \left[ \frac34 -  \frac{(\gamma_q+1)H(u,x_0, \rho r_n)}{(\rho r_n)^{N-1} \|u\|_{x_0, \rho r_n}^2}\right]
\end{align*}
for every $n$ large, which together with \eqref{30 11 2} implies that
\begin{equation}\label{H su norm non-deg}
\frac{H(u,x_0, \rho r_n)}{(\rho r_n)^{N-1}\|u\|_{x_0, \rho r_n}^2 } \ge \frac1{2(\gamma_q+1)} >0
\end{equation}
for every $n$ large. 

We are ready to reach a contradiction. Since $\{v_{\rho r_n}\}$ is bounded in $H^1(B_1)$ by definition and $\alpha_{\rho r_n} \to 0$ by \eqref{30 11 2}, by compactness of Sobolev embedding and of the trace operator we have that up to a subsequence $v_{\rho r_n} \to \bar v$ weakly in $H^1(B_1)$, strongly in $L^2(B_1)$ and strongly in $L^2(S_1)$, and the limit $\bar v$ is harmonic in $B_1$. By Lemma \ref{lem: weak to strong}, we deduce that the convergence is in fact strong in $H^1_{\loc}(B_1)$. Now, on one side, by estimate \eqref{H su norm non-deg}
\[
H(\bar v,0,1) = \lim_{n \to \infty} H(v_{\rho r_n}, 0,1) = \lim_{n \to \infty} \frac{H(u,x_0, \rho r_n)}{(\rho r_n)^{N-1}\|u\|_{x_0, \rho r_n}^2 } \ge C>0,
\]
so that $\bar v \not \equiv 0$ in $B_1$. But on the other side, having assumed \eqref{30 11 1} we also deduce that
\[
\|\bar v\|_{0,1/\rho} = \lim_{n \to \infty} \|v_{\rho r_n}\|_{0,1/\rho} = \lim_{n \to \infty} \frac{\|u\|_{x_0,r_n}}{\|u\|_{x_0,\rho r_n}} = 0,
\]
which forces $\bar v \equiv 0$ in $B_{1/\rho}$. Clearly this is not possible by the unique continuation property of harmonic functions.
\end{proof}

We finally recall the following result.

\begin{lemma}[Lemma 4.1, \cite{Wei01}]\label{lem: hom wei}
Let $\alpha -1 \in \N$, $w \in H^1(B_\rho)$ be a harmonic function in $B_\rho$, and assume that $\cO(w,0) \ge \alpha$. Then
\begin{equation}\label{ineq hom}
\frac{1}{\rho^{N-2}}\int_{B_\rho} |\nabla w|^2 \ge \frac{\alpha}{\rho^{N-1}} \int_{S_\rho} w^2,
\end{equation}
and equality implies that $w$ is homogeneous of degree $\alpha$ in $B_\rho$.
\end{lemma}

We observe that the lemma is written in a slightly different form in \cite{Wei01}; instead of $\cO(w,0)  \ge \alpha$ it is required that $D^j w(0) = 0$ for any multi-index $j \in \N^N$ with $0 \le |j| \le \alpha -1$. Since harmonic functions are smooth, this means precisely that $0$ is a zero of $w$ with order at least $\alpha$. Moreover, the lemma is stated in $B_1$, but it holds in the above form by scaling.

\medskip

We can now proceed with the:

\begin{proof}[Conclusion of the proof of Theorem \ref{thm: blow-up '}]
Due to the non-degeneracy, Theorem \ref{thm: non-deg}, up to a subsequence we have two possibilities:
\[
\text{either} \quad \frac{\|u\|_{x_0,r_n}}{r_n^{\gamma_q}} \to \ell  \in (0,+\infty), \quad \text{or} \quad \frac{\|u\|_{x_0,r_n}}{r_n^{\gamma_q}} \to +\infty.
\]
Suppose at first that $\|u\|_{r_n}/r_n^{\gamma_q} \to \ell$ finite. Due to Lemma \ref{lem: v_r bdd}, the sequence $\{v_{r_n}\}$ is bounded in $H^1_{\loc}(\R^N)$; thus, compactness argument and Lemma \ref{lem: weak to strong}, together with a diagonal selection, imply that up to a subsequence $v_{r_n} \to v$ strongly in $H^1_{\loc}(\R^N)$. The equation for $v_{r_n}$ and elliptic estimates imply that actually $v_{r_n} \to v$ in $C^{1,\alpha}_{\loc}(\R^N)$.

It is clear that the limit $v$ solves \eqref{eq limit 1 12} for $\mu= \ell^{-2/\gamma_q}$, and $v \not \equiv 0$ since, by strong $H^1(B_1)$-convergence, $\|v\|_{0,1}=1$. It remains to prove that $v$ is homogeneous, and to this end we appeal to Corollary \ref{cor: W 2 mon}: for any $t>0$ and $n$ large enough
\begin{align*}
 W_{\gamma_q,2}(v_{r_n},0,t) & =  \frac{r_n^{2\gamma_q}}{\|u\|_{x_0,r_n}^2}\left( \frac{1}{(r_n t)^{N-2+ 2\gamma_q}} \int_{B_{r_n t}(x_0)} |\nabla u|^2 - \frac{\gamma_q}{(r_n t)^{N-1+2 \gamma_q}} \int_{S_{r_n t}(x_0)} u^2 \right)  \\
 & - \frac{2 \alpha_{r_n}}{q (r_n t)^{N-2+2\gamma_q}} \cdot \frac{r_n^{2 \gamma_q-2}}{\|u\|_{x_0,r_n}^q} \int_{B_{r_n t}(x_0)} F_{\lambda_+,\lambda_-}(u) \\
 & = \frac{r_n^{2\gamma_q}}{\|u\|_{x_0,r_n}^2}W_{\gamma_q,2}(u,x_0,r_n t),
\end{align*}
where we used the definition of $\alpha_{r_n}$.
Passing to the limit as $n \to \infty$, we infer by $C^{1,\alpha}_{\loc}(\R^N)$ convergence that 
\[
W_{\gamma_q,2}(v,0,t) = \lim_{n \to \infty} \frac{r_n^{2\gamma_q}}{\|u\|_{x_0,r_n}^2}W_{\gamma_q,2}(u,x_0,r_n t) = \frac{1}{\ell^2} W_{\gamma_q,2}(u,x_0,0^+)
\]
for every $t \in \R$ (even though this is not necessary, we observe that, since $v \in H^1_{\loc}(\R^N)$, the previous equality implies that $W_{\gamma_q,2}(u,x_0,0^+) \in \R$ in case $\|u\|_{r_n}/r_n^{\gamma_q} \to \ell$ finite). As the right hand side is independent of $t$, we proved that $W_{\gamma_q,2}(v,0,\cdot)$ is constant, and hence the $2/(2-q)$-homogeneity of $v$ follows by Corollary \ref{cor: W 2 mon}. 

The case $\|u\|_{r_n}/r_n^{\gamma_q} \to +\infty$ requires some extra care. The $C^{1,\alpha}_{\loc}(\R^N)$ convergence $v_{r_n} \to v$ can be proved as before. Since now $\alpha_{r_n} \to 0$, the limit function $v$ is harmonic in $\R^N$, and as before we have that
\begin{equation}\label{1 12 3}
W_{\gamma_q,2}(v_{r_n},0,t) = \frac{r_n^{2\gamma_q}}{\|u\|_{x_0,r_n}^2}W_{\gamma_q,2}(u,x_0,r_n t).
\end{equation}
The problem is that, passing to the limit in $n$, we cannot show that the right hand side tends to a quantity independent of $t$. Therefore, the homogeneity must be proved in a different way. Let $r_0 <\dist(x_0,\pa B_1)$ be arbitrarily chosen. By \eqref{1 12 3} and the monotonicity of $W_{\gamma_q,2}(u,x_0,\cdot)$, for every $n$ sufficiently large
\begin{align*}
W_{\gamma_q,2}(v_{r_n},0,t) \le  \frac{r_n^{2\gamma_q}}{\|u\|_{x_0,r_n}^2} W_{\gamma_q,2}(u,x_0,r_0),
\end{align*}
whence
\begin{align*}
\frac{1}{t^{N-2}}\int_{B_t} |\nabla v_{r_n}|^2 \le \frac{t^{2 \gamma_q} r_n^{2\gamma_q}}{\|u\|_{x_0,r_n}^2}  W_{\gamma_q,2}(u,x_0,r_0) + \frac{2 \alpha_{r_n}}{qt^{N-2}} \int_{B_t} F_{\lambda_+,\lambda_-}(v_{r_n}) + \frac{\gamma_q}{t^{N-1}} \int_{S_t} v_{r_n}^2.
\end{align*}
Now $|W_{\gamma_q,2}(u,x_0,r_0)|<+\infty$ since $u \in H^1(B_{r_0}(x_0))$, $\alpha_{r_n} \to 0$, and hence passing to the limit as $n \to \infty$ we obtain
\begin{equation}\label{eq hom}
\frac{1}{t^{N-2}}\int_{B_t} |\nabla v|^2 \le \frac{\gamma_q}{t^{N-1}} \int_{S_t} v^2,
\end{equation}
for every $t>0$. On the other hand, by Proposition \ref{prop: upper} and the fact that $\cO(v_{r_n},0) = \cO(u,x_0) = \gamma_q$ for every $n$, we have that $\cO(v,0) \ge \gamma_q$; hence, Lemma \ref{lem: hom wei} establishes that inequality \eqref{ineq hom} holds for $v$ with $\alpha = \gamma_q$. This means that in \eqref{eq hom} equality must hold, and hence $v$ is $2/(2-q)$-homogeneous, as desired. 
\end{proof}

\begin{remark}
Instead of Proposition \ref{prop: upper}, for $q=1$ we could simply use the $C^{1,\alpha}$ convergence of $v_{r_n} \to v$, which directly yields $\cO(v,0)>1$. 
\end{remark}

Having established Theorem \ref{thm: blow-up '}, we can easily prove Theorems \ref{thm: very strong V}, \ref{thm: non-deg V} and \ref{thm: blow-up}. They are straightforward corollaries of the following:

\begin{proposition}\label{prop: eq orders}
For every $x_0 \in B_1$ and $r \in(0,\dist(x_0,\pa B_1)$, it results that
\[
0 < \liminf_{r \to 0^+} \frac{H(u,x_0,r)}{r^{N-1} \|u\|_{x_0,r}^2} \le 1.
\]
In particular, if $x_0 \in Z(u)$, then $\cO(u,x_0) = \cV(u,x_0)$.
\end{proposition}

\begin{proof}
The upper estimate directly follows from the definitions of $H$ and of $\|\cdot\|_{x_0,r}$, so we can focus on the lower estimate. In case alternative ($i$) of Proposition \ref{thm: blow-up pre} holds, the thesis can be checked by direct computations. Thus, suppose by contradiction that alternative ($ii$) holds, and that for a sequence $0<r_n \to 0^+$
\begin{equation}\label{H fratto norma}
 \frac{H(u,x_0,r_n)}{r_n^{N-1} \|u\|_{x_0,r_n}^2} \to 0.
\end{equation}
Since $\cO(u,x_0) = 2/(2-q)$, the blow-up sequence $\{v_{r_n}\}$ converges, as $n \to \infty$ and up to a subsequence, to a limit $\bar u$, homogeneous non-trivial solution to \eqref{eq limit 1 12} in $\R^N$. But on the other hand, by \eqref{H fratto norma} we deduce that 
\[
\int_{S_1} \bar u^2 = \lim_{n \to \infty} \int_{S_1} v_{r_n}^2 = 0.
\]
By homogeneity, this implies that $\bar u \equiv 0$ in $\R^N$, a contradiction.
\end{proof}

\begin{proof}[Proof of Theorems \ref{thm: very strong V}, \ref{thm: non-deg V} and \ref{thm: blow-up}]
Theorems \ref{thm: very strong V} and \ref{thm: non-deg V} follow by Theorems \ref{thm: very strong}, \ref{thm: non-deg} and Proposition \ref{prop: eq orders}.

Theorem \ref{thm: blow-up} follows by Propositions \ref{thm: blow-up pre} and \ref{prop: eq orders} and Theorem \ref{thm: blow-up '}.
\end{proof}

\section{Dimension estimate and structure of the nodal set}\label{sec: structure}

Having established Theorem \ref{thm: blow-up}, we can proceed with the estimate on the Hausdorff dimension of $\Sigma(u)$ in Theorem \ref{thm: Hausdorff}. We shall apply \cite[Theorem 8.5]{Che}, a variant of the classical Federer's dimension reduction principle (for which we refer to \cite[Appendix A]{Sim}).\footnote{The result in \cite{Che} is stated in a very general form. Simpler versions, closer to what we really need, are Proposition 4.5 in \cite{Che2} and Theorem 4.6 in \cite{TavTer}. In particular, we shall apply \cite[Theorem 4.6]{TavTer} with the local uniform convergence replaced by the local $C^{1,\alpha}$ convergence.}

\begin{proof}[Proof of Theorem \ref{thm: Hausdorff} - Hausdorff dimension of the singular set]
Let $\alpha \in (0,1)$, let
\[
\cF:=\left\{ v \in C^{1,\alpha}(\R^N) \setminus \{0\}\left| \begin{array}{l} -\Delta v = \mu \Big(\lambda_+ (v^+)^{q-1}- \lambda_-(v^-)^{q-1}\Big) \quad \text{in $B_\rho$} \\ \text{for some $\rho > 2$ and some $\mu \ge 0$} \end{array}\right.\right\}, 
\]
endowed with $C^{1,\alpha}_{\loc}$ convergence, let $\cC$ be the class of all the relatively closed subsets of $B_1$, and $\Sigma: \cF \to \cC$ be defined by 
\[
\Sigma(v) := \{x \in B_1: v(x) =  |\nabla v(x)| = 0\}.
\]
The family $\cF$ is closed under scalings and translations.
The existence of a non-trivial homogeneous blow-up follows by Theorem \ref{thm: blow-up} and by known results regarding harmonic functions, as well as the singular set assumption (in order to prove the singular set assumption, we can directly appeal to the $C^{1,\alpha}_{\loc}$ convergence of the blow-ups). Thus, \cite[Theorem 8.5]{Che} is applicable, and implies that there exists an integer $0 \le d \le N-1$ such that
\[
\dim_{\cH} (\Sigma(v)) \le d \qquad \text{for every $v \in \cF$}.
\]
Moreover, there exists a $d$-dimensional linear subspace $E \subset \R^N$, and a $\alpha$-homogeneous solution $v \in \cF$ to 
\[
-\Delta v = \mu \Big(\lambda_+ (v^+)^{q-1}- \lambda_-(v^-)^{q-1}\Big) \quad \text{in the whole space $\R^N$}
\]
(for some $\alpha>0$ and $\mu \ge 0$) such that $\{v=|\nabla v|=0\}=E$, and 
\begin{equation}\label{d-dim}
\frac{u(x_0+\lambda x)}{\lambda^\alpha} = u(x) \qquad \text{for every $x_0 \in E$ and $\lambda>0$};
\end{equation}
identity \eqref{d-dim} means that $v$ is $\alpha$-homogeneous with respect to all the points in $E$, and hence, up to a rotation, it depends only on $N-d$ variables. Let us suppose by contradiction that $d=N-1$. Then, without loss of generality, we can suppose that $v(x_1,\dots,x_N) = w(x_1)$ for a function $w \in H^1_{\loc}(\R) \cap L^\infty_{\loc}(\R)$ with $\{w=w'=0\} = \{x_1=0\}$. Notice that $w \in C^{1,\alpha}(\R^N)$ by elliptic regularity, and hence $w$ is a $C^{1,\alpha}(\R)$ weak solution to 
\begin{equation}\label{1 d pr}
\begin{cases}
-w'' = \mu \big(\lambda_+ (w^+)^{q-1}- \lambda_-(w^-)^{q-1}\big) & \text{in $\R$} \\
w(0) = w'(0) = 0.
\end{cases}
\end{equation}
In case $\mu=0$, it is clear that necessarily $w \equiv 0$, by the uniqueness of solutions to the above Cauchy problem. If $\mu \neq 0$ and $1 \le q<2$, even though the right hand side of the equation for $w$ is not locally Lipschitz continuous, we still have that $w \equiv 0$ is the unique solution. This can be easily checked using the fact that the Hamiltonian function 
\[
\cH(w,w') := \frac{(w')^2}2 + \mu \lambda_+ \frac{(w^+)^q}q + \mu \lambda_-\frac{(w^-)^q}q
\] 
is constant along solutions to \eqref{1 d pr}\footnote{If $q =1$, we have that $\cH(w,w')$  is absolutely continuous, and the fundamental theorem of calculus yields $\frac{d}{dt} \cH(w,w')(t)=0$ a.e.}, and the only level curve of $\cH$ crossing the origin of the phase plane is the constant trajectory $0$. Thus, in both cases we reach a contradiction with the fact that $\{w=w'=0\}= \{x_1=0\}$, and this implies that $d \le N-2$, as desired.
\end{proof}

In order to complete the proof of Theorem \ref{thm: Hausdorff}, we still have to show that in the $2$-dimensional case $\Sigma(u)$ is discrete. We start with the preliminary observation that $v$ is a global $\alpha$-homogeneous solution of \eqref{eq} in $\R^N$ if and only if $u(r,\theta) = r^\alpha \varphi(\theta)$ with $\alpha =2/(2-q)=: \gamma_q$, and 
\begin{equation}\label{eq phi}
-\Delta_{\theta} \varphi - \underbrace{\gamma_q( N-2 + \gamma_q )}_{=:\lambda_{N,q}} \varphi =\lambda_+ (\varphi^+)^{q-1}- \lambda_-(\varphi^-)^{q-1} \quad \text{on $\S^{N-1}$},
\end{equation}
where $\Delta_\theta$ denotes the Laplace-Beltrami operator on $\S^{N-1}$. 

\begin{proposition}\label{prop: disc}
If $v$ is a non-trivial global $\gamma_q$-homogeneous solution to \eqref{eq mu} in $\R^2$ for some $\mu \ge 0$, then $\Sigma(v) \cap S_1= \emptyset$. 
\end{proposition}

\begin{proof}
Suppose by contradiction that there exists $p =(\cos{\bar \theta}, \sin{\bar \theta})\in S_1 \cap \Sigma(v)$. Up to a rotation, we can suppose that $\bar \theta=0$. Then $v$ vanishes together with its gradient in $p$, which in polar coordinates $v=r^{\gamma_q} \varphi$ yields $\varphi(0) = \varphi'(0) = 0$. Since $\varphi$ solves \eqref{eq phi} on the unit circle $\S^1$ with $\lambda_{N,q} = \gamma_q^2$ and with $\lambda_\pm$ replaced by $\mu \lambda_\pm$, we have
\begin{equation}\label{eq phi contr}
\begin{cases}
-\varphi'' - \gamma_q^2 \varphi= \mu \big(\lambda_+ (\varphi^+)^{q-1}- \lambda_-(\varphi^-)^{q-1}\big)  & \text{in }\left[0,2\pi\right] \\
\varphi \text{ is $2\pi$-periodic} \\
 \varphi(0) = \varphi'(0) = 0.
\end{cases}
\end{equation}
But then, exactly as in the first part of the proof of Theorem \ref{thm: Hausdorff}, using the constancy of the Hamiltonian function along solutions to \eqref{eq phi contr} we deduce that $\varphi \equiv 0$, which is against the fact that $v \not \equiv 0$. 
\end{proof}

As a consequence:

\begin{proof}[Proof of Theorem \ref{thm: Hausdorff} - case $N=2$] 
We show that, if $N=2$, then $\Sigma(u)$ is discrete. To this end, we suppose by contradiction that there exists a sequence of points $\{x_n\} \subset \Sigma(u)$, with $x_n \to x_0$ in $\Sigma(u)$ as $n \to \infty$. For $r_n:= |x_n-x_0|$, we consider the blow-up sequence $u_{n}(x):= u(x_0+r_n x)/ (r_n^{1-N} H(u,x_0,r_n))^{1/2}$; by Theorem \ref{thm: blow-up}, $u_n \to \bar u$ in $C^{1,\alpha}_{\loc}(\R^N)$ for every $0<\alpha<1$, where $\bar u$ is either a non-trivial homogeneous harmonic polynomial, or a non-trivial $2/(2-q)$-homogeneous solution of \eqref{eq mu} in $\R^2$ (for some $\mu>0$); moreover, by the choice of $r_n$ and by $C^{1,\alpha}$ convergence, there exists $p \in S_1$ such that $p \in \Sigma(\bar u)$. Since we are in dimension $N=2$ and $\bar u$ is homogeneous and non-trivial, this is however not possible (see Proposition \ref{prop: disc}).
\end{proof}

Finally, with Theorem \ref{thm: Hausdorff} and Proposition \ref{prop: upper} in our hands, we can easily prove Theorem \ref{thm: struct sing}, completing our description of the structure of $\Sigma(u)$ is higher dimension.

\begin{proof}[Proof of Theorem \ref{thm: struct sing}]
By Proposition \ref{prop: upper} with $v_k = u$ and $\mu_k=1$ for every $k$, we deduce in particular that $x_0 \mapsto \mathcal{O}(u,x_0) =\cV(u,x_0)$ is upper semi-continuous. This implies that $\cT(u)$ is relatively closed in $\Sigma(u)$, which is relatively closed in $Z(u)$. 

For the countable $(N-2)$-rectifiability of $\cS(u)$, we follow the argument introduced in \cite{Han94}. Let $x_0 \in \mathcal{S}(u)$, and let $P_{x_0}$ be the leading polynomial at $x_0$, given by Proposition \ref{thm: blow-up pre}-($i$). We define the normalized leading polynomial
\[
\varphi_{x_0}(x):= \frac{P_{x_0}(x)}{ H(P_{x_0},0,1)^{1/2}},
\]
and introduce the blow-up family
\[
u_{x_0,r}(x):= \frac{u(x_0+rx)}{ \left( \frac{1}{r^{N-1}} \int_{S_r(x_0)} u^2 \right)^\frac12}.
\]
It is not difficult to check that $u_{x_0,r} \to \varphi_{x_0}$ in $W^{1,\infty}_{\loc}(\R^N)$ as $r \to 0$. At this point it is possible to apply step by step the proof of \cite[Theorem 2.1]{Han94} (using Proposition \ref{prop: upper} instead of \cite[Lemma 1.4]{Han94}), obtaining the desired result.
\end{proof}

\section{Multiplicity of global homogeneous solutions}\label{sec: global}

In this section we prove Theorem \ref{thm: multiple}. As already observed, $v$ is a global homogeneous solution of \eqref{eq} in $\R^2$ if and only if $u(r,\theta) = r^{\gamma_q} \varphi(\theta)$, with $\gamma_q=2/(2-q)$ and $\varphi$ $2\pi$-periodic solution to
\begin{equation}\label{eq phi N=2}
-\varphi'' - \gamma_q^2 \varphi =\lambda_+ (\varphi^+)^{q-1}- \lambda_-(\varphi^-)^{q-1} \quad \text{in $[0,2\pi]$}.
\end{equation}
Let $\bar k$ denote the minimum positive integer greater than or equal to $2 \gamma_q$, and, for any fixed $k > \bar k$, let us consider $T=T_k:=2\pi/k$. We prove the existence of a $2\pi$-periodic solution to \eqref{eq phi N=2} having exactly $2k$ zeros in $[0,2\pi)$. Since this construction can be carried on for any $k>\bar k$, the existence part in Theorem \ref{thm: multiple} follows.

For $0<t \le T$, the first eigenvalue of the Laplace-Beltrami operator $-\Delta_\theta = -d^2/(d \theta)^2$ on the arc $(0,t)$ is $\pi^2/t^2 > k^2/4> \gamma_q^2$. We consider the problem
\begin{equation}\label{eq phi k}
\begin{cases}
-\varphi'' - \gamma_q^2 \varphi= \lambda_+ |\varphi|^{q-2} \varphi & \text{in }\left(0,t\right) \\
\varphi>0 & \text{in }\left(0,t\right)
 \\ \varphi(0) = 0 = \varphi\left(t\right).
\end{cases}
\end{equation}
In order to show that \eqref{eq phi k} has a solution, we address the minimization of the functional $J_{(0,t)}^+ : H^1_0(0,t) \to \R$ defined by 
\[
J_{(0,t)}^+(\varphi):= \int_0^{t} \left(\frac12 \left(\varphi'\right)^2 - \frac{\gamma_q^2}2 \varphi^2 - \frac{\lambda_+}{q} |\varphi|^q\right).
\]

\begin{lemma}
With the previous choice of $T=2\pi/k$, for every $0<t<T$ there exists a unique non-negative minimizer of $J_{(0,t)}^+$ in $H_0^1(0,t)$, denoted by $\varphi_+(\cdot\,,t)$. Moreover, $\varphi_+(\cdot\,,t)>0$ in $(0,t)$ and $J_{(0,t)}^+(\varphi_+(\cdot\,,t)) <0$.
\end{lemma}

\begin{proof}
By Sobolev embedding, it is clear that $J_t^+$ is weakly lower semi-continuous, and by the Poincar\'e inequality we deduce that
\begin{equation}\label{stima J}
\begin{split}
J_{(0,t)}^+(\varphi) &\ge \frac{1}{2}\left(1-\frac{t^2 \gamma_q^2}{\pi^2}\right) \int_0^{t} \left(\varphi ' \right)^2 - \frac{\lambda_+ t^\frac{2-q}{2} }{q}  \left(\int_0^{t} \varphi^2\right)^\frac{q}{2} \\
& \ge \frac{1}{2}\left(1-\frac{t^2 \gamma_q^2}{\pi^2}\right) \int_0^{t} \left(\varphi ' \right)^2 - \frac{\lambda_+ t^\frac{2+q}{2} }{q \pi^q}  \left(\int_0^{t} \left(\varphi ' \right)^2\right)^\frac{q}{2},
\end{split}
\end{equation}
and the coefficient $(1-t^2 \gamma_q^2/\pi^2)$ is strictly positive since $0< t \le T$. Thus, $J_{(0,t)}^+$ is bounded from below and coercive. Thus, $J_{(0,t)}^+$ is bounded from below and coercive. Since $H_0^1(0,t)$ is weakly closed, the direct method of the calculus of variations implies the existence of a minimizer $\bar \varphi$, which solves the first equation in \eqref{eq phi k} together with the boundary condition. It remains to show that $\bar \varphi$ is positive. We can first suppose that $\bar \varphi \ge 0$, since if $\bar \varphi$ is a minimizer, then the same holds also for $|\bar \varphi|$. Now the strong maximum principle implies that either $\bar \varphi >0$, or $\bar \varphi \equiv 0$, but the latter alternative can be easily ruled out observing that $J_{(0,t)}^+(\bar \varphi)<0$: indeed, for every $\varphi \in H_0^1(0,t)$, one has $J_t^+(s\varphi) <0$ for $s>0$ small enough. 

Lastly, we prove the uniqueness of the non-negative minimizer. Suppose by contradiction that there exists two different minimizers $\bar \varphi \ge 0$ and $\tilde \varphi \ge 0$. They have to be positive in $(0,t)$ by the strong maximum principle. We claim that 
\begin{equation}\label{order min}
\text{either $\bar \varphi<\tilde \varphi$, or $\bar \varphi> \tilde \varphi$, in $(0,t)$}.
\end{equation} 
If not, we have that the graphs of $\bar \varphi$ and of $\tilde \varphi$ intersect in a point $\theta_1 \in (0,t)$, and $\bar \varphi \not \equiv \tilde \varphi$; then necessarily $\bar \varphi'(\theta_1) \neq \tilde \varphi'(\theta_1)$, by uniqueness for the Cauchy problem associated with the equation for $\bar \varphi$. Also, without loss of generality we can assume that $J_{(0,\theta_1)}^+(\bar \varphi) \ge J_{(0,\theta_1)}^+(\tilde \varphi)$. Since on the other hand $J_{(0,t)}^+(\bar \varphi) = J_{(0,t)}^+(\tilde \varphi)$, we infer that $J_{(\theta_1,t)}^+(\bar \varphi) \le J_{(\theta_1,t)}^+(\tilde \varphi)$, and therefore we can construct a third minimizer
\[
\hat \varphi(\theta):= \begin{cases} \tilde \varphi & \text{if $\theta \in [0,\theta_1]$} \\
\bar \varphi & \text{if $\theta \in (\theta_1,t]$}.  \end{cases}
\]
As a minimizer, $\hat \varphi$ is a smooth solution to \eqref{eq phi k}, but on the other hand by construction $\hat \varphi$ is not differentiable in $\theta_1$, a contradiction. This proves the validity of \eqref{order min}.

Now we observe that, by \eqref{eq phi k}, 
\[
\int_0^t (\bar \varphi')^2 -\gamma_q^2 {\bar{\varphi}}^2 = \int_{0}^t \lambda_+ |\bar \varphi|^q,
\]
and as a consequence
\[
J_{(0,t)}^+(\bar \varphi) = \lambda_+\left(\frac12-\frac1q\right) \int_0^t |\bar \varphi|^q.
\]
Analogously 
\[
J_{(0,t)}^+(\tilde \varphi) = \lambda_+\left(\frac12-\frac1q\right) \int_0^t |\tilde \varphi|^q,
\]
but then we obtain a contradiction between \eqref{order min} and the fact that $J_{(0,t)}^+(\tilde \varphi) = J_{(0,t)}^+(\bar \varphi)$. This finally shows that the non-negative minimizer is unique.
\end{proof}

The notation $\varphi_+'(\theta,t)$ will always be used to denote the derivative with respect to $\theta$.

\begin{lemma}\label{lem: bound min}
There exists $C>0$ such that 
\[
\|\varphi_+(\cdot\,,t)\|_{H_0^1(0,t)}^2 :=  \int_0^t \left(\varphi_+'(\cdot\,,t) \right)^\frac12 \le C t^\frac{2+q}{2-q}
\]
for every $t \in (0,T]$. 
\end{lemma}

\begin{proof}
By \eqref{stima J}, we infer that 
\begin{align*}
0 & > J_{(0,t)}^+(\varphi_+(\cdot\,,t)) \ge \frac{1}{2}\left(1-\frac{t^2 \gamma_q^2}{\pi^2}\right) \int_0^{t} \left(\varphi_+(\cdot\,,t)' \right)^2 - \frac{\lambda_+ t^\frac{2+q}{2} }{q \pi^q}  \left(\int_0^{t} \left(\varphi_+(\cdot\,,t)' \right)^2\right)^\frac{q}{2} \\
& \ge \underbrace{\frac{1}{2}\left(1-\frac{T^2 \gamma_q^2}{\pi^2}\right)}_{=: \bar C} \int_0^{t} \left(\varphi_+(\cdot\,,t)' \right)^2 - \frac{\lambda_+ t^\frac{2+q}{2} }{q \pi^q}  \left(\int_0^{t} \left(\varphi_+(\cdot\,,t)' \right)^2\right)^\frac{q}{2}.
\end{align*}
Notice that $\bar C$ is independent of $t \in (0,T]$ and, by the choice of $T=2\pi/k$ with $k > \bar k$ it results $\bar C>0$. Therefore we deduce that 
\[
 \int_0^{t} \left(\varphi_+(\cdot\,,t)' \right)^2 \le \frac{\lambda_+ t^\frac{2+q}{2} }{q \bar C \pi^q}   \left(\int_0^{t} \left(\varphi_+(\cdot\,,t)' \right)^2\right)^\frac{q}{2},
\]
whence the thesis follows.
\end{proof}

We now show the continuous dependence of the minimizers with respect to $t$. To be more precise, for any $t \in (0,T]$ let us define 
\[
\phi_+(\theta,t):= \frac1{t} \varphi_+(t\theta,t).
\]
It is clear that $\phi_+ \in H_0^1(0,1) \cap H^2(0,1)$, that $\phi_+(\cdot\,,t)$ minimizes 
\[
\tilde{J}_{(0,t)}^+(\phi) := \int_0^1 \left( \frac12 (\phi')^2 - \frac{\gamma_q^2 t^2}2 \phi^2 - \frac{\lambda_+ t^{q}}{q} |\phi|^q\right)
\]
in $H_0^1(0,1)$, and that 
\begin{equation}\label{eq phi k 1}
\begin{cases}
-\big(\phi_+(\cdot\,,t)\big)'' -t^2 \gamma_q^2 \phi_+(\cdot\,,t) = \lambda_+ t^q \big(\phi_+(\cdot\,,t)\big)^{q-1} & \text{in $(0,1)$} \\
\phi_+(\cdot\,,t)>0 & \text{in $(0,1)$} \\
\phi_+(0,t) = 0 = \phi_+(1,t).
\end{cases}
\end{equation}
The advantage of working with $\{\phi_+(\cdot\,,t): t \in (0,T]\}$ stays in the fact that this is a family of functions defined on the same interval. For future convenience, we define $\phi_+(\cdot\,,0)$ as the constant function $0$.

\begin{lemma}\label{lem: cont dep min}
If $t \to \bar t \in [0,T]$, then $\phi_+(\cdot\,,t) \to \phi_+(\cdot\,,\bar t)$ in $C^1([0,1])$.
\end{lemma}

\begin{proof}
By Sobolev embedding, it is sufficient to show that $\phi_+(\cdot\,,t_n) \to \phi_+(\cdot\,,\bar t)$ in $H^2(0,1)$. As a preliminary observation, we observe that by definition and Lemma \ref{lem: bound min}
\begin{equation}\label{norma a 0}
\|\phi_+(\cdot\,,t)\|_{H_0^1(0,1)}^2 = \frac{1}{t}\|\varphi_+(\cdot\,,t)\|_{H_0^1(0,t)}^2 \le C t^{\frac{2q}{2-q}}
\end{equation}
for every $t \in (0,T]$. 

Suppose at first that $\bar t \in (0,T]$. Up to a subsequence $\phi_+(\cdot\,,t_n) \weak \bar \phi$ weakly in $H_0^1(0,1)$. By \eqref{eq phi k 1}, the convergence is in fact strong in $H^2(0,1)$, and $\bar \phi \ge 0$ in $[0,1]$ so that to complete the proof we have only to check that $\bar \phi= \phi_+(\cdot\,,\bar t)$. If this is not true, then by uniqueness of the non-negative minimizer
\begin{equation}\label{1 2 1}
\tilde{J}_{(0,\bar t)}^+(\bar \phi) > \tilde{J}_{(0,\bar t)}^+(\phi_+(\cdot\,,\bar t)).
\end{equation}
But the functional $\tilde{J}_{(0,t)}^+(\phi)$ is continuous with respect to both $t \in \R$ and $\phi \in H_0^1(0,1)$, and hence
\[
\tilde{J}_{(0,t_n)}^+(\phi_+(\cdot\,,t_n)) \to \tilde{J}_{(0,\bar t)}^+(\bar \phi), \quad \text{and} \quad \tilde{J}_{(0,t_n)}^+(\phi_+(\cdot\,,\bar t)) \to \tilde{J}_{(0,\bar t)}^+(\phi_+(\cdot\,,\bar t)).
\]
Combining with \eqref{1 2 1}, we deduce that for sufficiently large $n$ we have 
\[
\tilde{J}_{(0,t_n)}^+(\phi_+(\cdot\,,t_n)) > \tilde{J}_{(0,t_n)}^+(\phi_+(\cdot\,,\bar t)),
\]
in contradiction with the minimality of $\phi_+(\cdot\,,t_n)$. This shows that necessarily $\bar \phi =  \phi_+(\cdot\,,\bar t)$, as desired.

In case $\bar t=0$, the $H_0^1(0,1)$ strong convergence $\phi_+(\cdot\,,t_n) \to 0$ follows from \eqref{norma a 0}. Equation \eqref{eq phi k 1} yields also strong convergence in $H^2(0,1)$.
\end{proof}

We deduce the following:

\begin{corollary}\label{cor: cont der t}
The function $t \mapsto \varphi_+'(t^-,t)$ is continuous in $t \in (0,T)$, with 
\[
\lim_{t \to 0^+} \varphi_+'(t^-,t) = 0 \quad \text{and} \quad \lim_{t \to T^-} \varphi_+'(t^-,t) <0. 
\]
\end{corollary}

\begin{proof}
Since $\varphi_+'(t^-,t) = \phi_+'(1^-,t)$, the thesis follows directly by Lemma \ref{lem: cont dep min} (recall that $\phi_+(\cdot\,,0) \equiv 0$) and by the Hopf lemma.
\end{proof}

We can adapt the very same argument to produce a solution to 
\begin{equation}\label{eq phi k neg}
\begin{cases}
-\varphi'' - \gamma_q^2 \varphi= \lambda_- |\varphi|^{q-2} \varphi & \text{in }\left(t,T\right) \\
\varphi<0 & \text{in }\left(t,T\right)
 \\ \varphi(t) = 0 = \varphi(T),
\end{cases}
\end{equation}
characterized as the unique non-positive minimizer of the functional $J_{(t,T)}^- : H^1_0(t,T) \to \R$ defined by 
\[
J_{(t,T)}^-(\varphi):= \int_0^{t} \left(\frac12 \left(\varphi'\right)^2 - \frac{\gamma_q^2}2 \varphi^2 - \frac{\lambda_-}{q} |\varphi|^q\right).
\]
If we denote such a minimizer by $\varphi_-(\cdot\,,t)$, it is possible to use the same argument leading to Corollary \ref{cor: cont der t} in order to show that:

\begin{corollary}\label{cor: cont der t 2}
The function $t \mapsto \varphi_-'(t^+,t)$ is continuous in $t \in (0,T)$, with 
\[
\lim_{t \to 0^+} \varphi_-'(t^+,t) < 0 \quad \text{and} \quad \lim_{t \to T^-} \varphi_-'(t^+,t) =0. 
\]
\end{corollary}

We are now ready for the:

\begin{proof}[Conclusion of the proof of Theorem \ref{thm: multiple}]
We consider the function 
\[
\Psi:(0,T) \to \R, \qquad \Psi(t):=  \varphi_+'(t^-,t) - \varphi_-'(t^+,t).
\]
By Corollaries \ref{cor: cont der t} and \ref{cor: cont der t 2}, it results that $\Psi$ is continuous in $(0,T)$, with $\Psi(t)>0$ for $t$ close to $0$ and $\Psi(t)<0$ for $t$ close to $T$. Therefore, there exists $\bar t \in (0,T)$ with $\Psi(\bar t) = 0$. Let us consider the function
\[
\tilde \varphi(\theta):= \begin{cases}  \varphi_+(\theta,\bar t) & \text{if $\theta \in [0,\bar t]$} \\  \varphi_-(\theta,\bar t) & \text{if $\theta \in [\bar t,T]$}. \end{cases}
\]
By construction, $\tilde \varphi$ is solution to \eqref{eq phi N=2} in $(0,T) \setminus \{\bar t\}$, and moreover is of class $C^1$, since $\Psi(\bar t)=0$. Then, $\tilde \varphi$ is an $H^2(0,1)$ solution to \eqref{eq phi N=2} in $(0,T)$, vanishes in $0$ and $T$, and has exactly one interior zero in $\bar t$. It follows that the energy function
\[
\frac{(\tilde \varphi')^2}2 + \frac{\gamma_q^2}{2} \tilde{\varphi}^2+ \lambda_+ \frac{(\tilde{\varphi}^+)^q}q + \lambda_-\frac{(\tilde{\varphi}^-)^q}q
\]
is constant in $\theta$, and in particular $(\tilde{\varphi}'(0^+))^2 = (\tilde{\varphi}'(T^-))^2$, whence 
\begin{equation}\label{der extre}
\tilde{\varphi}'(0^+)=\tilde{\varphi}'(T^-).
\end{equation} 
Recalling that $T=2\pi/k$ with $k \in \N$, this condition allows us to extend $\tilde{\varphi}$ on the whole unit circle, letting
\[
\tilde{\tilde{\varphi}}(\theta) := \begin{cases} \tilde{\varphi}(\theta) & \text{if $\theta \in \left[0,\frac{2\pi}k\right]$} \\  \tilde{\varphi}\left(\theta+ \frac{2\pi}{k}\right) & \text{if $\theta \in \left[\frac{2\pi}{k},\frac{4\pi}{k}\right]$} \\
\cdots \\
\tilde{\varphi}\left(\theta+ \frac{2(k-1)\pi}{k}\right) & \text{if $\theta \in \left[\frac{2(k-1)\pi}{k},2\pi\right]$}.
\end{cases}
\]
By \eqref{der extre} the new function $\tilde{\tilde{\varphi}}$ is of class $C^1([0,2\pi])$, is $2\pi$-periodic, and solve \eqref{eq phi N=2} in the whole unit circle. This completes the existence part in Theorem \ref{thm: multiple}. 

Let now $u_k$ the global homogeneous solution to \eqref{eq} associated with $\tilde{\tilde{\varphi}}$. It is clear that $\cO(u_k,0) = \gamma_q$ for every $k$, by homogeneity. Equality $N_q(u_k,0,1)=\gamma_q$ can be checked directly in the following way: for any $\gamma_q$-homogeneous solution $u = r^{\gamma_q} \varphi(\theta)$ to \eqref{eq}, passing to polar coordinates we have that
\[
N_q(u,0,1) = \frac{\frac{1}{N-2+2 \gamma_q} \int_{\S^{N-1}} |\nabla_\theta \varphi|^2 + \gamma_q^2 \varphi^2 - F_{\lambda_+,\lambda_-}(\varphi)}{ \int_{\S^{N-1}} \varphi^2}.
\]
Multiplying \eqref{eq phi} with $\varphi$ itself and integrating, we obtain
\[
\int_{\S^{N-1}} |\nabla_\theta \varphi|^2 - F_{\lambda_+,\lambda_-}(\varphi)= \lambda_{N,q} \int_{\S^{N-1}} \varphi^2,
\]
whence equality $N_q(u,0,1) = \gamma_q$ follows straightforwardly. This holds in particular for $u=u_k$, for any $k$.
\end{proof}


\begin{thebibliography}{10}

\bibitem{AltPhil}
H.~W. Alt and D.~Phillips.
\newblock A free boundary problem for semilinear elliptic equations.
\newblock {\em J. Reine Angew. Math.}, 368:63--107, 1986.

\bibitem{AnShWe1}
J.~Andersson, H.~Shahgholian, and G.~S. Weiss.
\newblock Uniform regularity close to cross singularities in an unstable free
  boundary problem.
\newblock {\em Comm. Math. Phys.}, 296(1):251--270, 2010.

\bibitem{AnShWe2}
J.~Andersson, H.~Shahgholian, and G.~S. Weiss.
\newblock On the singularities of a free boundary through {F}ourier expansion.
\newblock {\em Invent. Math.}, 187(3):535--587, 2012.

\bibitem{AnShWe3}
J.~Andersson, H.~Shahgholian, and G.~S. Weiss.
\newblock The singular set of higher dimensional unstable obstacle type
  problems.
\newblock {\em Atti Accad. Naz. Lincei Rend. Lincei Mat. Appl.},
  24(1):123--146, 2013.

\bibitem{AnWe}
J.~Andersson and G.~S. Weiss.
\newblock Cross-shaped and degenerate singularities in an unstable elliptic
  free boundary problem.
\newblock {\em J. Differential Equations}, 228(2):633--640, 2006.

\bibitem{Ber}
L.~Bers.
\newblock Local behavior of solutions of general linear elliptic equations.
\newblock {\em Comm. Pure Appl. Math.}, 8:473--496, 1955.

\bibitem{Bon}
L.~P. Bonorino.
\newblock Regularity of the free boundary for some elliptic and parabolic
  problems. {II}.
\newblock {\em Comm. Partial Differential Equations}, 26(3-4):355--380, 2001.

\bibitem{CafFri79}
L.~A. Caffarelli and A.~Friedman.
\newblock The free boundary in the {T}homas-{F}ermi atomic model.
\newblock {\em J. Differential Equations}, 32(3):335--356, 1979.

\bibitem{CafFri85}
L.~A. Caffarelli and A.~Friedman.
\newblock Partial regularity of the zero-set of solutions of linear and
  superlinear elliptic equations.
\newblock {\em J. Differential Equations}, 60(3):420--433, 1985.

\bibitem{Car}
T.~Carleman.
\newblock Sur un probl\`eme d'unicit\'e pur les syst\`emes d'\'equations aux
  d\'eriv\'ees partielles \`a deux variables ind\'ependantes.
\newblock {\em Ark. Mat., Astr. Fys.}, 26(17):9, 1939.

\bibitem{CheNabVal}
J.~Cheeger, A.~Naber, and D.~Valtorta.
\newblock Critical sets of elliptic equations.
\newblock {\em Comm. Pure Appl. Math.}, 68(2):173--209, 2015.

\bibitem{Che2}
X.-Y. Chen.
\newblock On the scaling limits at zeros of solutions of parabolic equations.
\newblock {\em J. Differential Equations}, 147(2):355--382, 1998.

\bibitem{Che}
X.-Y. Chen.
\newblock A strong unique continuation theorem for parabolic equations.
\newblock {\em Math. Ann.}, 311(4):603--630, 1998.

\bibitem{DonFef}
H.~Donnelly and C.~Fefferman.
\newblock Nodal sets of eigenfunctions on {R}iemannian manifolds.
\newblock {\em Invent. Math.}, 93(1):161--183, 1988.

\bibitem{FoSh}
M.~Fotouhi and H.~Shahgholian.
\newblock A semilinear {PDE} with free boundary.
\newblock {\em Nonlinear Anal.}, 151:145--163, 2017.

\bibitem{GaLi86}
N.~Garofalo and F.-H. Lin.
\newblock Monotonicity properties of variational integrals, {$A_p$} weights and
  unique continuation.
\newblock {\em Indiana Univ. Math. J.}, 35(2):245--268, 1986.

\bibitem{GaLi87}
N.~Garofalo and F.-H. Lin.
\newblock Unique continuation for elliptic operators: a geometric-variational
  approach.
\newblock {\em Comm. Pure Appl. Math.}, 40(3):347--366, 1987.

\bibitem{Han94}
Q.~Han.
\newblock Singular sets of solutions to elliptic equations.
\newblock {\em Indiana Univ. Math. J.}, 43(3):983--1002, 1994.

\bibitem{Han00}
Q.~Han.
\newblock Schauder estimates for elliptic operators with applications to nodal
  sets.
\newblock {\em J. Geom. Anal.}, 10(3):455--480, 2000.

\bibitem{HanHarLin}
Q.~Han, R.~Hardt, and F.~Lin.
\newblock Geometric measure of singular sets of elliptic equations.
\newblock {\em Comm. Pure Appl. Math.}, 51(11-12):1425--1443, 1998.

\bibitem{HarSim}
R.~Hardt and L.~Simon.
\newblock Nodal sets for solutions of elliptic equations.
\newblock {\em J. Differential Geom.}, 30(2):505--522, 1989.

\bibitem{HoMiTa}
S.~Hofmann, M.~Mitrea, and M.~Taylor.
\newblock Singular integrals and elliptic boundary problems on regular
  {S}emmes-{K}enig-{T}oro domains.
\newblock {\em Int. Math. Res. Not. IMRN}, (14):2567--2865, 2010.

\bibitem{KocTat}
H.~Koch and D.~Tataru.
\newblock Carleman estimates and unique continuation for second order parabolic
  equations with nonsmooth coefficients.
\newblock {\em Comm. Partial Differential Equations}, 34(4-6):305--366, 2009.

\bibitem{Lin}
F.-H. Lin.
\newblock Nodal sets of solutions of elliptic and parabolic equations.
\newblock {\em Comm. Pure Appl. Math.}, 44(3):287--308, 1991.

\bibitem{MoWe}
R.~Monneau and G.~S. Weiss.
\newblock An unstable elliptic free boundary problem arising in solid
  combustion.
\newblock {\em Duke Math. J.}, 136(2):321--341, 2007.

\bibitem{NabVal}
A.~Naber and D.~Valtorta.
\newblock Volume estimates on the critical sets of solutions to elliptic
  {PDE}s.
\newblock {\em Comm. Pure Appl. Math.}, 70(10):1835--1897, 2017.

\bibitem{PeShUr}
A.~Petrosyan, H.~Shahgholian, and N.~Uraltseva.
\newblock {\em Regularity of free boundaries in obstacle-type problems}, volume
  136 of {\em Graduate Studies in Mathematics}.
\newblock American Mathematical Society, Providence, RI, 2012.

\bibitem{Ru}
A.~R\"uland.
\newblock Unique continuation for sublinear elliptic equations based on
  {C}arleman estimates.
\newblock Preprint arxiv: 1801.05563.

\bibitem{ShUrWe}
H.~Shahgholian, N.~Uraltseva, and G.~S. Weiss.
\newblock The two-phase membrane problem---regularity of the free boundaries in
  higher dimensions.
\newblock {\em Int. Math. Res. Not. IMRN}, (8):Art. ID rnm026, 16, 2007.

\bibitem{ShWe}
H.~Shahgholian and G.~S. Weiss.
\newblock The two-phase membrane problem---an intersection-comparison approach
  to the regularity at branch points.
\newblock {\em Adv. Math.}, 205(2):487--503, 2006.

\bibitem{Sim}
L.~Simon.
\newblock {\em Lectures on geometric measure theory}, volume~3 of {\em
  Proceedings of the Centre for Mathematical Analysis, Australian National
  University}.
\newblock Australian National University, Centre for Mathematical Analysis,
  Canberra, 1983.

\bibitem{SoTeforth}
N.~Soave and S.~Terracini.
\newblock The nodal set of solutions to some elliptic problems: singular
  nonlinearities.
\newblock In preparation.

\bibitem{SoWe}
N.~Soave and T.~Weth.
\newblock The unique continuation property of sublinear equations.
\newblock Preprint arXiv:1707.07463.

\bibitem{TavTer}
H.~Tavares and S.~Terracini.
\newblock Regularity of the nodal set of segregated critical configurations
  under a weak reflection law.
\newblock {\em Calc. Var. Partial Differential Equations}, 45(3-4):273--317,
  2012.

\bibitem{Ura01}
N.~N. Uraltseva.
\newblock Two-phase obstacle problem.
\newblock {\em J. Math. Sci. (New York)}, 106(3):3073--3077, 2001.
\newblock Function theory and phase transitions.

\bibitem{Wei01}
G.~S. Weiss.
\newblock An obstacle-problem-like equation with two phases: pointwise
  regularity of the solution and an estimate of the {H}ausdorff dimension of
  the free boundary.
\newblock {\em Interfaces Free Bound.}, 3(2):121--128, 2001.

\end{thebibliography}

\end{document}